\pdfoutput=1
\documentclass[11pt,oneside,reqno]{amsart}

\usepackage{amsmath, amssymb, amsthm, amsfonts, mathrsfs, mathtools}
			
\setbox0=\hbox{$+$}
\newdimen\plusheight
\plusheight=\ht0
\def\+{\;\lower\plusheight\hbox{$+$}\;}

\newdimen\minusheight
\minusheight=\ht0
\def\-{\;\lower\minusheight\hbox{$-$}\;}

\setbox0=\hbox{$\cdots$}
\newdimen\cdotsheight
\cdotsheight=\plusheight
\def\cds{\lower\cdotsheight\hbox{$\cdots$}}
\usepackage{tikz}
\usetikzlibrary{matrix,decorations.pathreplacing}

\newtheorem{question}{Question}
\newtheorem{conjecture}{Conjecture}

\theoremstyle{definition}

\usepackage{graphicx, epsfig}

\theoremstyle{definition}

\makeatletter
\renewcommand\contentsnamefont{\bfseries}
\def\@starttoc#1#2{\begingroup
\setTrue{#1}%
\par\removelastskip\vskip\z@skip
\@startsection{}\@M\z@{\linespacing\@plus\linespacing}%
{.5\linespacing}{
\contentsnamefont}{#2}%
\ifx\contentsname#2%
\else \addcontentsline{toc}{section}{#2}\fi
\makeatletter
\@input{\jobname.#1}%
\if@filesw
\@xp\newwrite\csname tf@#1\endcsname
\immediate\@xp\openout\csname tf@#1\endcsname \jobname.#1\relax
\fi
\global\@nobreakfalse \endgroup
\addvspace{32\p@\@plus14\p@}%
\let\tableofcontents\relax
}
\def\contentsname{Contents}
\def\l@section{\@tocline{2}{.5ex}{0mm}{5pc}{}}
\def\l@subsection{\@tocline{2}{0pt}{2em}{5pc}{}}
\def\l@subsubsection{\@tocline{2}{0pt}{3em}{5pc}{}}
\def\l@subsubsubsection{\@tocline{2}{0pt}{4em}{5pc}{}}
\makeatother

\usepackage[strict]{changepage}
\makeatother

\numberwithin{equation}{section}
\theoremstyle{plain}
\newtheorem{theorem}{Theorem}[section]

\newtheorem{corollary}[theorem]{Corollary}
\newtheorem{proposition}[theorem]{Proposition}
\newtheorem{definition}[theorem]{Definition}
\newtheorem{lemma}[theorem]{Lemma}
\newtheorem{claim}[theorem]{Claim}
\usepackage{titlesec}

\usepackage{graphicx, epsfig}
\theoremstyle{definition}
\usepackage{graphicx} 
\usepackage{color} 
\usepackage{amsthm}
\usepackage{changepage}
\usepackage{blkarray}
\usepackage{transparent} 
\usepackage{titlesec}
\usepackage{changepage}
\usepackage{lipsum}
\makeatletter

\titleformat{\section}{\normalfont\LARGE\bfseries}
\titleformat{\subsubsection}{\normalfont\large\bfseries}

\renewcommand\section{\@startsection{section}{1}{\z@}%
                                  {-3.5ex \@plus -1ex \@minus -.2ex}%
                                 {2.3ex \@plus.2ex}%
                                 {\normalfont\Large\bfseries}}
\renewcommand\subsection{\@startsection{subsection}{1}{\z@}%
                                  {-3.5ex \@plus -1ex \@minus -.2ex}%
                                 {2.3ex \@plus.2ex}%
                                 {\normalfont\normalsize\bfseries}}
\renewcommand\subsubsection{\@startsection{subsubsection}{1}{\z@}%
                                  {-3.5ex \@plus -1ex \@minus -.2ex}%
                                 {2.3ex \@plus.2ex}%
                                 {\normalfont\normalsize\bfseries}}
\titlespacing*{\paragraph}
{0pt}{3.25ex plus 1ex minus .2ex}{1.5ex plus .2ex}

\usepackage{titlesec}
\usepackage{hyperref}

\titleclass{\subsubsubsection}{straight}[\subsection]

\newcounter{subsubsubsection}[subsubsection]
\renewcommand\thesubsubsubsection{\thesubsubsection.\arabic{subsubsubsection}}

\newcommand{\newparallel}{\mathrel{\mathpalette\new@parallel\relax}}
\newcommand{\new@parallel}[2]{%
  \begingroup
  \sbox\z@{$#1T$}
  \resizebox{!}{\ht\z@}{\raisebox{\depth}{$\m@th#1/\mkern-5mu/$}}%
  \endgroup}
\makeatletter
\newcommand{\notparallel}{%
  \mathrel{\mathpalette\not@parallel\relax}%
}
\newcommand{\not@parallel}[2]{%
  \ooalign{\reflectbox{$\m@th#1\smallsetminus$}\cr\hfil$\m@th#1\newparallel$\cr}%
}
\titleformat{\subsubsubsection}
  {\normalfont\normalsize}{\thesubsubsubsection}{1em}{}
\titlespacing*{\subsubsubsection}
{0pt}{3.25ex plus 1ex minus .2ex}{1.5ex plus .2ex}
\makeatother
\usepackage{graphicx} 
\usepackage{color} 
\usepackage{transparent} 
\usepackage{geometry} 
\begin{document}

\title{On anomalous subvarieties of holonomy varieties of hyperbolic $3$-manifolds} 
\author{BoGwang Jeon}

\begin{abstract}
The goal of this paper is to explore the interplay between two seemingly distinct fields. More precisely, let $\mathcal{M}$ be an $n$-cusped hyperbolic $3$-manifold with rationally independent cusp shapes, and $\mathcal{X}$ be its holonomy variety. We study the structure of anomalous subvarieties of $\mathcal{X}$, a concept originating in arithmetic geometry, and relate it to various geometric properties of $\mathcal{M}$.

First, we show that every maximal anomalous subvariety of $\mathcal{X}$ containing the identity is its subvariety of codimension $1$ which arises by keeping one cusp of $\mathcal{M}$ complete. 

Second, we show that, if $\mathcal{X}$ is degenerated by its anomalous subvarieties (i.e., $\mathcal{X}^{oa}=\emptyset$), then $\mathcal{M}$ has cusps which are, while keeping some other cusps of it complete, strongly geometrically isolated from the rest. 

Finally, we completely classify and characterize the case $\mathcal{X}^{oa}=\emptyset$ for the holonomy variety $\mathcal{X}$ of any $2$-cusped hyperbolic $3$-manifold.
\end{abstract} 
\maketitle
\tableofcontents

\section{Introduction}
\subsection{The motivating conjecture}
Before stating the main results, we will briefly review some motivating questions, theorems, and conjectures related to our research. As the target readers of this paper are primarily topologists, this is mainly to provide them with some familiarity with the topic `\textit{anomalous subvariety}' given in the title, which belongs to a different field. However, as readers may understand, covering a wide range of the topic in a couple of paragraphs is nearly impossible, and the task would never be complete. Instead, for more detailed information about the topic, we would like to recommend a nice introductory book by U. Zannier \cite{zan}. 

Let $\mathbb{G}^n:=(\overline{\mathbb{Q}}^*)^n$ or $(\mathbb{C^*})^n$. By an \textit{algebraic subgroup} $H$ in $\mathbb{G}^n(:=(X_1, \dots, X_n))$, we mean the set of solutions of equations of type $X_1^{a_1}\cdots X_n^{a_n}=1$ where $a_i\in \mathbb{Z}$ for $1\leq i\leq n$. An \textit{algebraic coset} is defined to be a translate $gH$ of some algebraic subgroup $H$ by some $g \in \mathbb{G}^n$. In particular, if $g$ is a torsion point(that is, a point whose coordinates are roots of unity), $gH$ is referred to as a \textit{torsion coset}. 

In the 1960s, S. Lang proposed the following question \cite{lang}:
\begin{question}[Lang]\label{23080701}
If $\mathcal{X}$ is an irreducible algebraic curve in $\mathbb{G}^2$ containing infinitely many torsion points, what can be said about $\mathcal{X}$?
\end{question}

First, we note the condition imposed on $\mathcal{X}$ is highly unusual. Since a torsion point in $\mathbb{G}^2$ is defined by two equations by definition, its containment in a plane curve $\mathcal{X}$ implies the point is an intersection point of three distinct equations in $\mathbb{G}^2$, which is very unlikely.  

The above question was soon answered by Ihara, Serre, and Tate as follows \cite{lang}:
\begin{theorem}[Ihara-Serre-Tate (first version)]\label{23080301}
If $\mathcal{X}\subset \mathbb{G}^2(:=(X_1, X_2))$ contains infinitely many torsion points, then $\mathcal{X}$ is a torsion coset. In other words, $\mathcal{X}$ is defined by an equation of the following form $X_1^{a_1}X_2^{a_2}=\zeta$ where $\zeta$ is a root of unity. 
\end{theorem}

In other words, a generic algebraic curve contains only finitely many torsion points. 

The question mentioned above, along with the theorem that followed, has stimulated extensive subsequent research, becoming one of the central topics in number theory \cite{pila,zan}. One natural consideration is the generalization of Question \ref{23080701} to higher-dimensional algebraic varieties. However, in this case, the problem turned out to be far more intricate and necessarily required the incorporation of methodologies from algebraic geometry. The final formulation of the theorem, which took a few more decades to be settled, is stated below. It was proven by M. Laurent \cite{lau} and independently later by Sarnak-Adams \cite{SA}:	
\begin{theorem}[Laurent, Sarnak-Adams (first version)]\label{23080303}
Let $\mathcal{X}$ be an algebraic variety in $\mathbb{G}^n$. Then the set of torsion points in $\mathcal{X}$ all lie and are Zariski dense in a finite number of torsion cosets contained in $\mathcal{X}$. 
\end{theorem}

That is to say, if the torsion points are Zariski dense over $\mathcal{X}$, then $\mathcal{X}$ itself is a torsion coset. Otherwise, they lie along a finite number of proper subvarieties of $\mathcal{X}$, each contained in a torsion coset. In some ways, Theorem \ref{23080303} is faithful to the spirit of the Bombieri-Lang conjecture, saying the set of rational points on an algebraic variety of general type is not Zariski dense but contained in its proper algebraic subvarieties. 

Note that a torsion point is an algebraic subgroup of dimension $0$. Thus, the existence of a torsion point on a given variety $\mathcal{X}$ is, in a certain sense, regarded as a highly stringent requirement. We therefore slightly relax this criterion and consider the scenario where $\mathcal{X}$ intersects with an algebraic subgroup of positive dimension. To further explain this, we define the following, which is due to Bombieri-Masser-Zannier \cite{za}: 

\begin{definition} \label{20040802}
Let $\mathcal{X}$ be an algebraic variety and $H$ be an algebraic coset in $\mathbb{G}^n$. An irreducible component $\mathcal{Y}$ of $\mathcal{X}\cap H$ is called an anomalous subvariety of $\mathcal{X}$ if it is of positive dimension and satisfies
\begin{equation}\label{20031301}
\dim \mathcal{Y}>\dim H +\dim \mathcal{X}-n.
\end{equation}
In particular, if $H$ is a torsion coset, then $\mathcal{Y}$ is called a torsion anomalous subvariety of $\mathcal{X}$. Also $\mathcal{Y}$ is said to be maximal if it is not contained in a strictly larger anomalous subvariety of $\mathcal{X}$.
\end{definition}

By a standard fact in intersection theory (for instance, see Proposition 3.28 in \cite{mum}), we always have
\begin{equation*}
\dim \mathcal{Y}\geq \dim H+\dim \mathcal{X} -n
\end{equation*} 
and, generically, the equality holds when both $\mathcal{X}$ and $H$ are in general position. Hence an anomalous subvariety of $\mathcal{X}$ arises as an exceptional phenomenon, occurring when $\mathcal{X}$ intersects with an algebraic coset of $\mathbb{G}^n$ in an unexpected way. 

For $\dim \mathcal{X}=1$, since the only non-trivial subvariety of $\mathcal{X}$ is $\mathcal{X}$ itself, $\mathcal{X}$ is anomalous if and only if it is contained in some algebraic coset. Also if an algebraic coset $H$, as described in Theorem \ref{23080303}, is entirely contained in $\mathcal{X}$, then $\mathcal{X}\cap H=H$, indicating any component $\mathcal{Y}$ of $\mathcal{X}\cap H$ satisfies \eqref{20031301} so is an anomalous subvariety of $\mathcal{X}$. However, it is generally challenging to precisely determine the anomalous subvarieties of a given algebraic variety, unless the variety is of simple type.\footnote{For instance, if a variety is given as the product of two other varieties, then it contains infinitely many anomalous subvarieties. See Proposition \ref{20041602}.} To some extent, the above definition is of a conceptual nature rather than being practically applicable, similar to many other definitions in algebraic geometry. 

Employing Definition \ref{20040802}, Theorems \ref{23080301}-\ref{23080303} are rephrased as follows respectively:
\begin{theorem}[Ihara-Serre-Tate (second version)]\label{23080801}
If $\mathcal{X}\subset \mathbb{G}^2(:=(X_1, X_2))$ contains infinitely many torsion points, then $\mathcal{X}$ itself is torsion anomalous. 
\end{theorem}

\begin{theorem}[Laurent, Sarnak-Adams (second version)]\label{23080802}
Let $\mathcal{X}$ be an algebraic variety in $\mathbb{G}^n$. Then, possibly except for finitely many, the set of torsion points in $\mathcal{X}$ lies along a finite number of torsion anomalous subvarieties of $\mathcal{X}$. 
\end{theorem}

Given the restatements above, it now becomes evident that, in order to achieve a far broader generalization of Theorems  \ref{23080801}-\ref{23080802}, one needs to make use of the notion of a `torsion anomalous subvariety' as defined in Definition \ref{20040802}. In particular, the generalization must also encompass the examination of the distribution of torsion anomalous subvarieties within a given variety. After undergoing several refinements, the final version, which synthesizes numerous previously established outcomes including the above ones, is conjecturally formulated as follows:

\begin{conjecture}[Zilber-Pink]\label{23080307}
For every irreducible variety $\mathcal{X}(\subset \mathbb{G}^n)$ defined over $\overline{\mathbb{Q}}$, there exists a finite set $\mathcal{T}$ of proper algebraic subgroups such that, for every algebraic subgroup $H$ and every component $\mathcal{Y}$ of $\mathcal{X}\cap H$ satisfying 
\begin{equation*}
\dim \mathcal{Y}>\dim H+\dim \mathcal{X} -n, 
\end{equation*} 
one has $\mathcal{Y}\subset T $ for some $T\in \mathcal{T}$. 
\end{conjecture}
Note that it is assumed $\dim \mathcal{Y}>0$ in Definition \ref{20040802}, but $\dim\mathcal{Y}=0$ is allowed in the Zilber-Pink conjecture (thus covering Theorems \ref{23080801}-\ref{23080802}).

The Zilber-Pink conjecture was proved for the curve case by G. Maurin \cite{mau}, varieties of codimension $2$ by Bombieri-Masser-Zannier \cite{za, BMZ1}, but is widely open for other cases \cite{pila, zan}.

The conjecture can be extended further to other contexts. For instance, we may replace $\mathbb{G}^n$ with an abelian variety, and ask a corresponding question. Implications of the conjecture have now spanned beyond the realms of number theory and algebraic geometry, touching on various other fields including logic and topology \cite{jeon5,jeon6,pila,zan,zil}. We will elaborate on this later from the perspective of 3-dimensional geometry and topology; see Section \ref{23080305}.

\subsection{Main results}\label{23081503}
To address the Zilber-Pink conjecture for $\mathcal{X}$, it is often essential to comprehend not only the underlying structure of torsion anomalous subvarieties but also that of the entire anomalous subvarieties in $\mathcal{X}$ (for instance, see Theorem \ref{19082401} below). The goal of this paper is to investigate the anomalous subvarieties of a special type of algebraic varieties, originating from low-dimensional topology, known as holonomy varieties of hyperbolic $3$-manifolds. The \textit{holonomy variety} $\mathcal{X}$ of a cusped hyperbolic $3$-manifold $\mathcal{M}$ is defined as the representation variety of the fundamental group of $\mathcal{M}$, with a particular choice of coordinates related to the geometric structures of the cusps of $\mathcal{M}$. The subject has been studied in great detail, as it provides much topological information about $\mathcal{M}$.

In this paper, we investigate the structure of the anomalous subvarieties of $\mathcal{X}$. Particularly, by imposing a certain condition on $\mathcal{X}$, we identify a necessary condition for $\mathcal{M}$ to degenerate $\mathcal{X}$ via its anomalous subvarieties, a condition that is related to a well-known geometric concept in the field. 

Our first main result is the following. Throughout the paper, the term `\textit{hyperbolic 3-manifold}' always refers to an orientable, complete hyperbolic 3-manifold of finite volume.
 
\begin{theorem}\label{19072301}
Let $\mathcal{M}$ be an $n$-cusped ($n\geq 2$) hyperbolic $3$-manifold having rationally independent cusp shapes and $\mathcal{X}$ be its holonomy variety. Then a maximal anomalous subvariety of $\mathcal{X}$ containing the identity is its subvariety of codimension $1$ obtained by keeping one cusp of $\mathcal{M}$ complete.  
\end{theorem}

Please see Definition \ref{20033001} for the precise meaning of the assumption on $\mathcal{M}$. Roughly, the assumption captures the algebraic manifestation of `non-symmetries' among the cusps of $\mathcal{M}$. We believe that this condition is sufficiently general and it will be further treated in more detail in Section \ref{23080201}.  

If a cusp of $\mathcal{M}$ allows the complete hyperbolic metric, two coordinate functions over $\mathcal{X}$ associated to the cusp are fixed by $1$ and they together determine an anomalous subvariety of $\mathcal{X}$ of codimension $1$. This is an anomalous subvariety of the simplest type possessed by every holonomy variety. 

On the other hand, if there exists a symmetry between two cusps, or even more strongly, if a manifold has a self-isometry sending one cusp to another, then this naturally leads to another distinct type of an anomalous subvariety on its holonomy variety (see Section \ref{23080305} for more details). Consequently, in simpler terms, Theorem \ref{19072301} tells us that when there are `no symmetries' among the cusps of a manifold, the holonomy variety has only an anomalous subvariety of the simplest type. 

Following \cite{za}, let $\mathcal{X}^{oa}$ be the complement of the set of anomalous subvarieties of $\mathcal{X}$. Then one can make progress towards Conjecture \ref{23080307} as follows \cite{hab}:  


\begin{theorem}[Bombieri-Habegger-Masser-Zannier]\label{19082401}
Let $\mathcal{X}$ be an $r$-dimensional irreducible variety in $\mathbb{G}^{r+s}$ defined over $\overline{\mathbb{Q}}$. Then 
\begin{equation*}
\bigcup_{\dim H=s-1}\mathcal{X}^{oa}\cap H
\end{equation*}
is finite. 
\end{theorem}

By the work of Bombieri-Masser-Zannier, it is known that $\mathcal{X}^{oa}$ is a Zariski open subset of $\mathcal{X}$ (Theorem \ref{struc}). Thus if $\mathcal{X}^{oa}\neq \emptyset$ and further $\mathcal{X}$ has only finitely many maximal anomalous subvarieties $\{\mathcal{Y}_i\}_{1\leq i\leq n}$, the Zilber-Pink conjecture for $\mathcal{X}$ is reduced to the same conjecture over $\{\mathcal{Y}_i\}_{1\leq i\leq n}$. Of course $\mathcal{X}^{oa}=\emptyset $ is also possible and Theorem \ref{19082401} tells us nothing in this case. 

Our second main result, stated below, concerns the condition $\mathcal{X}^{oa} = \emptyset$ for the holonomy variety $\mathcal{X}$ of a cusped hyperbolic $3$-manifold $\mathcal{M}$. Specifically, it provides geometric criteria on $\mathcal{M}$ that must be satisfied for $\mathcal{X}^{oa} = \emptyset$ to hold.

\begin{theorem}\label{19082701}
Let $\mathcal{M}$ be an $n$-cusped ($n\geq 2$) hyperbolic $3$-manifold having rationally independent cusp shapes and $\mathcal{X}$ be its holonomy variety. If $\mathcal{X}^{oa}=\emptyset$, then there exist cusps of $\mathcal{M}$, keeping some other cusps complete, strongly geometrically isolated (SGI) from the remaining cusps of $\mathcal{M}$. 
\end{theorem}

The term \textit{strong geometric isolation} (SGI) was first introduced by W. Neumann and A. Reid in \cite{rigidity}, and later further studied by D. Calegari \cite{calegari, calegari2}. Simply put, it means there exists a set of cusps of a manifold that moves independently without affecting the remaining cusps. In this case, its holonomy variety $\mathcal{X}$ is represented as the product of two varieties of lower dimensions, that is, $\mathcal{X}=\mathcal{X}_1\times \mathcal{X}_2$, and so one easily finds $\mathcal{X}^{oa}=\emptyset$. (See Proposition \ref{20041602}.) However, $\mathcal{X}^{oa}=\emptyset$ does not necessarily mean SGI but instead a slightly broader one according to Theorem \ref{19082701}. 

We postpone the definitions of SGI and its generalization as well as further discussions around them until Section \ref{SGI}. 

Theorem \ref{19082701} is an extension of the following theorem proved by the author:
\begin{theorem} \cite{jeon1}\label{20033003}
Let $\mathcal{M}$ be a $2$-cusped hyperbolic $3$-manifold having rationally independent cusp shapes. Then $\mathcal{X}^{oa}=\emptyset$ if and only if two cusps of $\mathcal{M}$ are SGI to each other. 
\end{theorem}

Since the condition on $\mathcal{M}$ described in Theorem \ref{19082701} or \ref{20033003} is generic, and manifolds exhibiting SGI phenomena on their cusps are quite rare \cite{calegari2}, it appears that $\mathcal{X}^{oa}\neq \emptyset$ is the most typical outcome, when $\mathcal{X}$ is given as the holonomy variety of a hyperbolic $3$-manifold. Moreover, we believe the condition on $\mathcal{M}$ in Theorem \ref{19082701} or \ref{20033003} may be removable; however, verifying this seems to be very difficult. Instead, we characterize $\mathcal{X}^{oa}=\emptyset$ for the $2$-cusped case completely as follows:  

\begin{theorem}\label{20042301}
Let $\mathcal{M}$ be a $2$-cusped hyperbolic $3$-manifold and $\mathcal{X}$ be its holonomy variety. If $\mathcal{X}^{oa}=\emptyset$, then either one of the following holds:
\begin{enumerate}
\item two cusps of $\mathcal{M}$ are SGI; or
\item there exists a two variable polynomial $f$ such that $\mathcal{X}$ is defined by 
\begin{equation}\label{200423041}
\begin{gathered}
f(M_1^{a}L_1^{b}M_2^{c}L_2^{d}, M_1^{md}M_2^{mb})=0, \quad f(M_1^{a}L_1^{b}M_2^{-c}L_2^{-d}, M_1^{md}M_2^{-mb})=0 
\end{gathered}
\end{equation}
for some $a,b,c,d\in \mathbb{Z}$ and $m\in \mathbb{Q}$ satisfying $mbd\neq 0$. 
\end{enumerate}
\end{theorem}
If two cusps of $\mathcal{M}$ are SGI, as mentioned earlier, $\mathcal{X}$ is the product of two algebraic curves. On the other hand, if $\mathcal{X}$ is given by equations as characterized in \eqref{200423041}, then it again becomes the product of two algebraic curves defined by $f(\tilde{M_1}, \tilde{L_1})=f(\tilde{M_2}, \tilde{L_2})=0$ in $\mathbb{G}^4(:=(\tilde{M_1}, \tilde{L_1}, \tilde{M_2}, \tilde{L_2}))$ under the following monoidal transformation:
\begin{equation*}
\tilde{M_1}:=M_1^{a}L_1^{b}M_2^{c}L_2^{d},\quad  \tilde{L_1}:=M_1^{md}M_2^{mb},\quad\tilde{M_2}:=M_1^{a}L_1^{b}M_2^{-c}L_2^{-d}, \quad \tilde{L_2}:=M_1^{md}M_2^{-mb}.
\end{equation*}
Consequently, $\mathcal{X}^{oa}=\emptyset$ implies the structure of $\mathcal{X}$ is directly or indirectly related to the product of two algebraic curves. As mentioned above, however, we doubt the existence of a manifold whose holonomy variety is of the form given in \eqref{200423041}, since this condition is highly restrictive. (In fact, as a byproduct of Theorem \ref{20042301}, we provide a straightforward criterion in Corollary \ref{23081005} that any manifold whose holonomy variety is not of the type given in \eqref{200423041} must satisfy. We have applied this criterion to more than a hundred manifolds in the SnapPy census \cite{CDGW}, and all of them satisfy it.)

Note that our main results provoke interesting topological consequences. For instance, using Theorems \ref{20042301}, we demonstrate a weak version of the Zilber-Pink conjecture for $\mathcal{X}\times \mathcal{X}$ and apply it to classify Dehn fillings of $\mathcal{M}$ \cite{jeon5,jeon6} (where $\mathcal{M}$ and $\mathcal{X}$ are the same as in Theorem \ref{20042301}). More generally, if one is able to fully resolve the Zilber-Pink conjecture for any holonomy variety as well as effectively determine the structure of its anomalous subvarieties, then it is expected that we can completely classify Dehn fillings of any cusped hyperbolic $3$-manifold. This would stand as the ultimate goal of the interplay between these two seemingly different fields, and the content of this paper should be acknowledged as a stepping stone towards this overarching theme. 

\subsection{Remarks on the proof strategy}

The general structure theorem, Theorem \ref{struc}, concerning anomalous subvarieties of arbitrary algebraic varieties, was established by Bombieri, Masser, and Zanier in [4]. We employ this theorem as a key tool in achieving our main results. Additionally, rather than directly considering the holonomy variety itself, we opt to work within an analytic framework. The approach involves taking the logarithm of each coordinate and is introduced in the context of the \textit{analytic holonomy set}, detailed in Section \ref{A-poly}. By adopting this perspective, we can fully leverage the local geometric characteristics it offers. For instance, the analytic holonomy set of a given hyperbolic $3$-manifold is recognized for exhibiting numerous symmetric properties, as outlined in Theorem \ref{potential}. These properties serve as a second key ingredient throughout the paper and are used in conjunction with Theorem \ref{struc}. Although the proofs rely primarily on relatively elementary ideas and techniques, chiefly from linear algebra, they are accompanied by a variety of interesting and new aspects which we believe were previously unexplored. 

\subsection{Acknowledgements}
We would like to thank Stephan Tillmann for many helpful comments on earlier drafts of this paper, and the anonymous referees for their careful reading and valuable suggestions, which have improved the overall quality and presentation of the paper.

This work was partially supported by the National Research Foundation of Korea (NRF) grant RS-2025-23323903.

\section{Background}\label{Pre}

In this section, we provide background knowledge on key topics mentioned in the previous section, such as the holonomy variety, anomalous subvarieties, and geometric isolation. Along the way, we also establish several key propositions that will be used later to prove the main results.

\subsection{Holonomy variety}\label{A-poly}
In this subsection, we introduce the holonomy variety of a cusped hyperbolic 3-manifold. We will not provide all the technical aspects, as they are not required for proving the main theorems. Instead, we suggest that readers refer to, for instance, \cite{Ahbijit} or \cite{CCGLS} for a more detailed description of the topic.

According to Thurston \cite{thu}, for a given $n$-cusped hyperbolic $3$-manifold $\mathcal{M}$, a geometric ideal triangulation $\mathcal{T}$ of it induces a so-called \textit{gluing variety} $G(\mathcal{T})$ of $\mathcal{M}$. The variety represents the necessary conditions for how the tetrahedra in $\mathcal{T}$ are glued together along their edges to yield a hyperbolic metric on $\mathcal{M}$. Roughly, $G(\mathcal{T})$ can also be seen as the set of all the possible hyperbolic structures on $\mathcal{M}$, or simply, the moduli space of $\mathcal{M}$. 

If $T_i$ is a torus cross-section of the $i^{\text{th}}$-cusp of $\mathcal{M}$ and $\mathfrak{m_i},\mathfrak{l_i}$ are the chosen meridian-longitude pair of $T_i$ $(1\leq i\leq n)$,  then each point of $G(\mathcal{T})$ gives rise to a (Euclidean) similarity structure on $T_i$, thus inducing the following \textit{holonomy map} 
\begin{equation*}
\pi_1(T_i)\longrightarrow \text{Aff}(\mathbb{C})=\{az+b\;:\;a\neq 0, b\in \mathbb{C}\}.
\end{equation*}
Consequently, the dilation components (i.e. derivatives) of the holonomies of $\mathfrak{m_i}$ and $\mathfrak{l_i}$ produce rational functions $M_i$ and $L_i$ respectively on $G(\mathcal{T})$. Now the \textit{holonomy variety} $\mathcal{X}$ of $\mathcal{M}$ is defined as the Zariski closure of the image under the following map:
\begin{equation*}
\xi\;:\;G(\mathcal{T})\longrightarrow (M_1, L_1, \dots, M_n, L_n).
\end{equation*}
The holonomy variety turns out to be independent of the triangulation $\mathcal{T}$, but depends only on the fundamental group of $\mathcal{M}$ as well as the chosen meridian-longitude pair of $T_i$. 

The holonomy variety, in general, may have several irreducible components; nevertheless, we are only interested in the so-called \emph{geometric component} of it. It is known that the geometric component of the holonomy variety of an $n$-cusped hyperbolic $3$-manifold $\mathcal{M}$ is an $n$-dimensional algebraic variety in $\mathbb{G}^{2n}\big(:=(M_1, L_1, \dots, M_n, L_n)\big)$ and contains $(1, \dots, 1)$ which gives rise to the complete hyperbolic metric structure of $\mathcal{M}$. We typically denote the component by $\mathcal{X}$ and, by abuse of notation, still term it as the holonomy variety of $\mathcal{M}$.  

The holonomy variety is known to possess many interesting symmetries. To describe them more effectively, it is convenient to work locally and analytically using logarithmic coordinates. More precisely, let
\begin{align}
u_i := \log M_i, \qquad v_i := \log L_i \qquad (1 \leq i \leq n). 
\end{align}
Then there exists a small neighborhood $\mathcal{N}_{\mathcal{X}}$ of a branch of $\mathcal{X}$ containing $(1,\dots,1)$ such that, in these logarithmic coordinates, the following statements hold on $\mathcal{N}_{\mathcal{X}}$ \cite{nz}:

\begin{theorem}[Neumann-Zagier] \label{potential}
(1)~$v_i=u_i\cdot\tau_i(u_1,\dots,u_n)$ where $\tau _i(u_1,\dots,u_n)$ is a holomorphic function with $\tau_i(0,\dots,0)=\tau_i\in \mathbb{C}\backslash\mathbb{R}$ ($1\leq i\leq n$).\\
(2)~There is a holomorphic function $\Phi(u_1,\dots,u_n)$ such that $v_i=\dfrac{1}{2}\dfrac{\partial \Phi}{\partial u_i}$ ($1\leq i\leq n$) and $\Phi(0,\dots,0)=0$.\\
(3)~$\Phi(u_1,\dots,u_n)$ is even in each argument and so its Taylor expansion is of the following form:
\begin{equation*}
\Phi(u_1,\dots,u_n)=(\tau_1u_1^2+\cdots+\tau_n u_n^2)+(c_{4,\dots,0}u_1^4+\cdots+c_{0,\dots,4}u_n^4)+\emph{(higher order)}.
\end{equation*}
\end{theorem}

We call $\tau_i$ the \emph{cusp shape} of $T_i$ with respect to $\mathfrak{m_i}, \mathfrak{l_i}$ and $\Phi(u_1,\dots,u_n)$ the \emph{Neumann-Zagier potential function} of $\mathcal{M}$ with respect to $\mathfrak{m_i}, \mathfrak{l_i}$ $(1\leq i\leq n)$. 
\begin{definition}
The complex manifold defined locally near 
$$(0,\dots,0)\in\mathbb{C}^{2n}(:=(u_1,v_1, \dots, u_n, v_n))$$ via the following holomorphic functions
\begin{equation*}
v_i=u_i\cdot\tau_i(u_1,\dots,u_n)\quad (1\leq i\leq n)
\end{equation*}
is called the analytic holonomy set of $\mathcal{M}$ and denoted by $\log\mathcal{X}$.  
\end{definition}

Clearly, by definition, $\log\mathcal{X}$ is locally biholomorphic to $\mathcal{N}_{\mathcal{X}}$. For $H$ given as an algebraic coset, $\mathcal{X}\cap H$ in general may have several irreducible components, but we focus solely on the component that touches $\mathcal{N}_{\mathcal{X}}$. Therefore, unless otherwise stated, when referring to $\mathcal{X}\cap H$, it always means that specific component. 

If an algebraic subgroup $H$ is explicitly given as
\begin{equation*}
\begin{gathered}
M_1^{a_{i1}}L_1^{b_{i1}} \cdots M_n^{a_{in}}L_n^{b_{in}} = 1\quad (1\leq i\leq m),
\end{gathered}
\end{equation*}
taking the logarithm of each coordinate, it is equivalent to
\begin{equation*}
\begin{gathered}
a_{i1}u_1 +b_{i1}v_1 +\cdots +a_{in}u_n +b_{in}v_n =0\quad (1\leq i\leq m).
\end{gathered}
\end{equation*}
Thus, locally near the identity, $\mathcal{X}\cap H$ is biholomorphic to the  manifold defined by
\begin{equation}\label{20042305}
\begin{gathered}
a_{i1}u_1 +b_{i1}u_1\tau_1(u_1,\dots,u_n)+\cdots+a_{in}u_n +b_{in}u_n\tau_n(u_1,\dots,u_n)=0\quad (1\leq i\leq m).
\end{gathered}
\end{equation}
Note that the dimension of $\mathcal{X}\cap H$ is attained by computing the rank of the Jacobian  
\begin{equation}\label{20040401}
\begin{gathered}
\left( \begin{array}{ccc}
a_{11}+b_{11}\tau_1 & \hdots & a_{1n}+b_{1n}\tau_n \\
       \vdots        & \ddots &        \vdots        \\
a_{m1}+b_{m1}\tau_1 & \hdots & a_{mn}+b_{mn}\tau_n 
\end{array} \right)
\end{gathered}
\end{equation}
of \eqref{20042305} at $(0, \dots , 0)$. 

As mentioned earlier, Theorem \ref{potential} will serve as a key ingredient in the proofs of the main theorems. Essentially, what we aim to do throughout the paper is to implement Theorem \ref{struc} in the context of holonomy varieties. To this end, we pass to the analytic setting outlined above and make use of Theorem \ref{potential}. This enables us to apply the results of Theorem \ref{struc}, which are somewhat theoretical, in a practical manner in the setting of holonomy varieties.\\

\noindent\textbf{Convention. }We represent the complex manifold in \eqref{20042305} by $\log (\mathcal{X}\cap H)$, call \eqref{20040401} the \textit{Jacobian} of $\log (\mathcal{X}\cap H)$ and often abbreviate it simply by 
$$\begin{array}{c}
( a_{ij}+b_{ij}\tau_j)_{1\leq i\leq m, \;1\leq j\leq n}.
\end{array}$$
Accordingly, the following matrix associated with $H$
$$
\left( \begin{array}{ccccc}
a_{11} & b_{11} & \hdots & a_{1n} &b_{1n} \\
       \vdots   & \vdots     & \ddots & \vdots  &     \vdots        \\
a_{m1} & b_{m1} & \hdots & a_{mn} & b_{mn} 
\end{array} \right)
$$
will be abbreviated as 
$$\begin{array}{c}
( a_{ij}\;\;b_{ij})_{1\leq i\leq m, \;1\leq j\leq n}.
\end{array}$$

The following proposition is elementary yet powerful, and it is frequently invoked throughout the paper. It also contains one of the core ideas from which this work originated. See also Proposition \ref{anomalous3}.

\begin{proposition}\label{anomalous2}
Let $\mathcal{X}$ be the holonomy variety of an $n$-cusped hyperbolic $3$-manifold and $H$ be an algebraic subgroup of codimension $m$ defined only in terms of the variables $M_i, L_i$ for $1\leq i\leq m$. That is, $H$ is defined by equation of the following form:
\begin{equation*}
M_1^{a_{i1}}L_1^{b_{i1}}\cdots M_m^{a_{im}}L_m^{b_{im}}=1, \quad (1\leq i\leq m). 
\end{equation*}
Then the rank of the Jacobian of $\log(\mathcal{X}\cap H)$ is $m$ if and only if
\begin{equation}\label{25070203}
\mathcal{X}\cap H=\mathcal{X}\cap (M_1=\cdots=M_m=1). 
\end{equation} 
\end{proposition}
\begin{proof}
Equivalently, we replace \eqref{25070203} with 
\begin{equation*}
\log (\mathcal{X}\cap H)=\log \mathcal{X}\cap (u_1=\cdots=u_m=0). 
\end{equation*} 
Since the dimension of $\log \mathcal{X}\cap (u_1=\cdots=u_m=0)$ is $n-m$, the ``if'' direction is clear. 

The ``only if'' direction follows from the obvious inclusion
\begin{equation*}
\log \mathcal{X}\cap (u_1=\cdots=u_m=0)\subset \log (\mathcal{X}\cap H),  
\end{equation*} 
combined with the implicit function theorem. 
\end{proof}
\subsection{Anomalous subvarieties on holonomy varieties}\label{23080305}
For a given cusped hyperbolic $3$-manifold $\mathcal{M}$, if its holonomy variety $\mathcal{X}$ is defined as above, anomalous subvarieties are readily detected from topological perspectives. For instance, by Theorem \ref{potential} (1), 
\begin{equation*}
u_i=0\Longleftrightarrow v_i=0
\end{equation*}
for each $i$ ($1\leq i\leq n$), and this implies   
\begin{equation}\label{20042501}
\big(\mathcal{X}\cap (M_i=1)\big)=\big(\mathcal{X}\cap (L_i=1)\big)=\big(\mathcal{X}\cap (M_i=L_i=1)\big).
\end{equation} 
Note that the variety in \eqref{20042501} is an $(n-1)$-dimensional subvariety of $\mathcal{X}$. Hence, if we let $H$ denote the algebraic subgroup defined by $M_i=L_i=1$ and $\mathcal{Y}$ denote the variety given in \eqref{20042501}, then $H, \mathcal{Y}$ and $\mathcal{X}$ satisfy the dimension condition in \eqref{20031301}; that is, $\mathcal{Y}$ is an anomalous subvariety of $\mathcal{X}$. 

Recall that $\mathcal{X}$ parametrizes deformations of hyperbolic structures (complete or incomplete) on $\mathcal{M}$ and the point $(M_1,\dots, L_n)=(1, \dots, 1)$ corresponds to the unique complete hyperbolic metric on $\mathcal{M}$. Therefore, in a geometric sense, one may view $\mathcal{Y}$ as the anomalous subvariety of $\mathcal{X}$ obtained by \textit{keeping the} $i^{\text{th}}$ \textit{cusp of} $\mathcal{M}$ \textit{complete}. 


Another source of anomalous subvarieties arises from a symmetry among the cusps of a manifold. For instance, let $\mathcal{M}$ be a $2$-cusped hyperbolic $3$-manifold having a self-isometry sending one cusp to the other. Moreover, assuming that the meridian $\mathfrak{m_1}$ (resp. longitude $\mathfrak{l_1}$) of the first cusp maps to the meridian $\mathfrak{m_2}$ (resp. longitude $\mathfrak{l_2}$) of the second cusp under the isometry, it is straightforward to verify that $u_1$ and $u_2$ are swapped when $v_1$ and $v_2$ are swapped. That is, 
$$v_1(u_2, u_1)=v_2(u_1, u_2)\quad  \text{and}\quad  v_2(u_2, u_1)=v_1(u_1, u_2).$$ Consequently, using the properties of $v_1$ and $v_2$ described in Theorem \ref{potential}, one has 
\begin{equation*}
u_1=u_2 \Longleftrightarrow v_1=v_2,
\end{equation*}
which further implies 
\begin{equation}\label{23081501}
\big(\mathcal{X}\cap (M_1=M_2)\big)=\big(\mathcal{X}\cap (L_1=L_2)\big)=\big(\mathcal{X}\cap (M_1=M_2, L_1=L_2)\big)
\end{equation}
is a $1$-dimensional anomalous subvariety of $\mathcal{X}$. In Section \ref{exam}, we provide a concrete example fitting into this scenario. 

Finally, when a cusped hyperbolic $3$-manifold $\mathcal{M}$ is realized as a covering space of another hyperbolic $3$-manifold $\mathcal{N}$, the holonomy variety of $\mathcal{M}$ contains an anomalous subvariety. Specifically, if $\mathcal{M}$ is a $2$-cusped hyperbolic $3$-manifold covering a $1$-cusped hyperbolic $3$-manifold $\mathcal{N}$ such that the meridian-longitude pair of each cusp of $\mathcal{M}$ maps to that of the cusp of $\mathcal{N}$, then identifying $\mathfrak{m_1}$ (resp. $\mathfrak{l_1}$) with $\mathfrak{m_2}$ (resp. $\mathfrak{l_2}$) yields a presentation of the fundamental group of $\mathcal{N}$. Accordingly,
\begin{equation}\label{25062501}
\mathcal{X}\cap (M_1=M_2, L_1=L_2)
\end{equation}
is a subvariety of $\mathcal{X}$ isomorphic to the holonomy variety of $\mathcal{N}$, which is of dimension $1$, and hence \eqref{25062501} is an anomalous subvariety of $\mathcal{X}$. Likewise, the holonomy varieties of manifolds in the commensurability class of $\mathcal{M}$ are interconnected within the framework of anomalous subvarieties.

The notion of anomalous subvariety plays a crucial role, particularly in the classification of Dehn fillings of a cusped hyperbolic $3$-manifold. For instance, the author, jointly with S. Oh, classifies Dehn fillings of a $1$-cusped hyperbolic $3$-manifold $\mathcal{M}$ via complex volume in \cite{JO}. It is further shown that a sequence of pairs of Dehn fillings of $\mathcal{M}$ with the same complex volume corresponds to an anomalous subvariety of $\mathcal{X} \times \mathcal{X}$. In other words, the points corresponding to these pairs of Dehn fillings always lie on this anomalous subvariety.

An analogue of the aforementioned result also holds for $2$-cusped manifolds \cite{jeon5, jeon6, JO1}, and it appears that it can be extended to any $n$-cusped hyperbolic $3$-manifold with $n \geq 3$. Consequently, studying and classifying anomalous subvarieties of $\mathcal{X}$ (and of $\mathcal{X} \times \mathcal{X}$) is of great importance in addressing concrete problems in hyperbolic geometry, beyond purely theoretical or aesthetic considerations.

\subsection{Rationally independent cusp shapes}\label{23080201}
As discussed above, let $\mathcal{M}$ be a $2$-cusped hyperbolic $3$-manifold allowing a self-isometry. In this case, if it is further assumed that  the isometry induces the following map
\begin{equation*}
\mathfrak{m_1}\rightarrow \mathfrak{m_2}^a\mathfrak{l_2}^b \quad \text{and}\quad \mathfrak{l_1}\rightarrow \mathfrak{m_2}^c\mathfrak{l_2}^d, 
\end{equation*}
then the two cusp shapes $\tau_1$ and $\tau_2$ of $\mathcal{M}$ are related by  
\begin{equation}\label{23080202}
\frac{1}{\tau_1}=\frac{a+b\tau_2}{c+d\tau_2}.
\end{equation}
On the other hand, for a given $2$-cusped hyperbolic $3$-manifold $\mathcal{M}$, if there are no $a,b,c,d\in \mathbb{Z}$ satisfying \eqref{23080202} for the two cusps $\tau_1, \tau_2$ of $\mathcal{M}$, that is an obstruction for the existence of a self-isometry sending one cusp to the other. Motivated by this, we present the following definition:
\begin{definition}
Let $\mathcal{M}$ be a $2$-cusped hyperbolic $3$-manifold with two cusp shapes $\tau_1, \tau_2$. We say $\mathcal{M}$ has rationally independent cusp shapes if $1,\tau_1, \tau_2$ and $\tau_1\tau_2$ are linearly independent over $\mathbb{Q}$.\footnote{Note that the definition is independent of the choice of $m_i, l_i$ ($1\leq i\leq 2$).}   
\end{definition}
The above definition applies quite generally. For instance, if $\tau_2\notin\mathbb{Q}(\tau_1)$, then $\tau_1$ and $\tau_2$ are rationally independent. 

Now we further extend the definition to any $n$-cusped ($n\geq 3$) hyperbolic $3$-manifold as follows: 
\begin{definition}\label{20033001}
Let $\mathcal{M}$ be an $n$-cusped manifold and $\tau_1, \dots, \tau_n$ be its cusp shapes. We say $\mathcal{M}$ has rationally independent cusp shapes if the elements in
\begin{equation*}
\{1\}\cup \{\tau_{i_1}\cdots\tau_{i_l}\text{ }|\text{ }1\leq i_1<\cdots<i_l\leq n\}\end{equation*}
are linearly independent over $\mathbb{Q}$.\footnote{Similar to the $2$-cusped case, the definition in this case is also independent of the choice of $m_i, l_i$ ($1\leq i\leq n$).} 
\end{definition}
For example, if $\mathcal{M}$ is a $3$-cusped hyperbolic $3$-manifold with cusp shapes $\tau_i$ ($1\leq i\leq 3$), then $\mathcal{M}$ has rationally independent cusp shapes when
\begin{equation*}
1, \quad \tau_1,\quad \tau_2, \quad \tau_3,\quad  \tau_1\tau_2,\quad  \tau_1\tau_3, \quad \tau_2\tau_3, \quad \tau_1\tau_2\tau_3
\end{equation*}
are linearly independent over $\mathbb{Q}$. Similar to the $2$-cusped case, if $\tau_{i+1}\notin\mathbb{Q}(\tau_1, \dots, \tau_i)$ for each $i$ ($1\leq i\leq n-1$), then $\tau_1, \dots \tau_n$ are rationally independent. 

According to W. Neumann's realization conjecture, every non-real number field can be realized as the trace field of some hyperbolic $3$-manifold. Given the fact that the cusp shapes are the elements of the trace field and, moreover, they typically generate the trace field \cite{NR2}, it appears that a generic cusped hyperbolic $3$-manifold has rationally independent cusp shapes. 

\subsection{Geometric isolation}\label{SGI}

The following is one of the equivalent definitions of SGI given in \cite{rigidity}:
\begin{definition}\label{isolation}
Let $\mathcal{M}$ be an $n$-cusped hyperbolic $3$-manifold. We say cusps $1,\dots,k$ are strongly geometrically isolated (SGI) from cusps $k+1,\dots,n$ 
if $v_1,\dots,v_k$ only depend on $u_1,\dots, u_k$ and not on $u_{k+1},\dots,u_n$.   
\end{definition}  
For the simplest case, if the analytic holonomy set $\mathcal{X}$ of a $2$-cusped hyperbolic $3$-manifold is defined by
\begin{equation*}
v_1=u_1\cdot\tau_1(u_1), \quad v_2=u_2\cdot\tau_2(u_2),
\end{equation*} 
that is, each $v_i$ ($i=1,2$) depends only on $u_i$, then two cusps of $\mathcal{M}$ are SGI from each other. 

The following proposition is proved easily. 
\begin{proposition}\label{20041602}
Let $\mathcal{M}$ be an $n$-cusped hyperbolic $3$-manifold and $\mathcal{X}$ be its holonomy variety. If $\mathcal{M}$ has cusps which are SGI from the rest, then $\mathcal{X}^{oa}=\emptyset$. 
\end{proposition}
\begin{proof}
Without loss of generality, suppose cusps $1, \dots, k$ ($k<n$) are SGI from the rest. 

First we claim that the holonomy variety $\mathcal{X}$ is of the form $\mathcal{X}_1\times \mathcal{X}_2$ in $\mathbb{G}^{2k}\times \mathbb{G}^{2n-2k}$. Consider the projection
\begin{equation*}
\begin{gathered}
\xi_1: \mathcal{X}\longrightarrow \mathbb{G}^{2k}(=(M_1,L_1, \dots, M_k, L_k)).
\end{gathered}
\end{equation*}
Then $\mathcal{X}_1:=\overline{\xi_1(\mathcal{X})}$ is an algebraic variety, which contains a local branch biholomorphic to the complex manifold defined by $v_i$ ($1\leq i\leq k$). As each $v_i$ ($1\leq i\leq k$) depends only on $u_1, \dots, u_k$ by the assumption, the dimension of the complex manifold\footnote{That is, the one defined by $v_i$ ($1\leq i\leq k$).} is $k$, which subsequently leads to $\dim \mathcal{X}_1= k$.\footnote{In fact, one can check $\mathcal{X}_1$ is equal to the projection of $\mathcal{X}\cap (M_1=\cdots=M_k=1)$ under $\xi_1$, which is a $k$-dimensional algebraic variety in $\mathbb{G}^{2k}$.} Similarly, by considering
\begin{equation*}
\begin{gathered}
\xi_2: \mathcal{X}\longrightarrow \mathbb{G}^{2n-2k}(=(M_{k+1},L_{k+1}, \dots, M_n, L_n)), 
\end{gathered}
\end{equation*}
we attain a variety $\mathcal{X}_2:= \overline{\xi_2(\mathcal{X})}$ of dimension $n-k$. As $\mathcal{X}\subset \mathcal{X}_1\times \mathcal{X}_2$ and $\dim \mathcal{X}=\dim (\mathcal{X}_1\times \mathcal{X}_2)=n$, it is  concluded that $\mathcal{X}=\mathcal{X}_1\times \mathcal{X}_2$. 
 
Now we claim $\mathcal{X}^{oa}=\emptyset$. Let $H_{(\xi_1^m, \dots, \xi_k^l)}$ be an algebraic coset defined by 
\begin{equation*}
M_1=\xi_1^m,\quad  L_1=\xi_1^l,\quad  \dots, \quad  M_k=\xi_k^m, \quad L_k=\xi_k^l
\end{equation*}
where $(\xi_1^m,\dots, \xi_k^l)\in\mathcal{X}_1$. Then clearly $H_{(\xi_1^m, \dots, \xi_k^l)}\cap \mathcal{X}$ is isomorphic to $\mathcal{X}_2$ for each $(\xi_1^m,\dots, \xi_k^l)\in\mathcal{X}_1$ and, as 
\begin{equation*}
\dim \mathcal{X}_2=n-k>\dim H_{(\xi_1^m, \dots, \xi_k^l)}+\dim \mathcal{X}-2n=(2n-2k)+n-2n=n-2k,
\end{equation*}
$H_{(\xi_1^m, \dots, \xi_k^l)}\cap \mathcal{X}$ is an anomalous subvariety of $\mathcal{X}$. Moreover, we have
\begin{equation*}
\bigcup_{(\xi_1^m,\dots, \xi_k^l)\in\mathcal{X}_1}(H_{(\xi_1^m, \dots, \xi_k^l)}\cap \mathcal{X})=\mathcal{X}, 
\end{equation*}
implying $\mathcal{X}^{oa}=\emptyset$. This completes the proof of the proposition. 
\end{proof}

However, the opposite direction of the above proposition is not true in general. To illustrate this, suppose that there exists a $3$-cusped hyperbolic $3$-manifold $\mathcal{M}$ whose Neumann-Zagier potential function $\Phi(u_1, u_2, u_3)$ is formally given as 
\begin{equation*}
\begin{gathered}
\Phi(u_1, u_2, u_3)=\sum^{\infty}_{i:even}a_{1i}u_1^{i}+\sum^{\infty}_{i:even}a_{2i}u_2^i+\sum^{\infty}_{i:even}a_{3i}u_3^{i}+\sum^{\infty}_{i, j:even}b_{i,j}u_1^{i}u_2^j+\sum^{\infty}_{i,j:even}c_{i,j}u_1^{i}u_3^j,
\end{gathered}
\end{equation*}
and thus 
\begin{equation}\label{19080601}
\begin{gathered}
v_1=\frac{1}{2}\Big(\sum^{\infty}_{i:even}ia_{1i}u_1^{i-1}+\sum^{\infty}_{i, j:even}ib_{i,j}u_1^{i-1}u_2^j+\sum^{\infty}_{i,j:even}ic_{i,j}u_1^{i-1}u_3^j\Big),\\
v_2=\frac{1}{2}\Big(\sum^{\infty}_{i:even}ia_{2i}u_2^{i-1}+\sum^{\infty}_{i, j:even}jb_{i,j}u_2^{j-1} u_1^i\Big),\;
v_3=\frac{1}{2}\Big(\sum^{\infty}_{i:even}ia_{3i}u_3^{i-1}+\sum^{\infty}_{i, j:even}jc_{i,j}u_3^{j-1}u_1^{i}\Big).
\end{gathered}
\end{equation}
For $\xi_1, \xi_2 \in \mathbb{C}$ sufficiently close to $0$, 
\begin{equation*}
\big(u_1=\xi_1,u_2=\xi_2, v_2=v_2(\xi_1, \xi_2)\big)\cap \log\mathcal{X}
\end{equation*}
is a $1$-dimensional analytic subset of $\log\mathcal{X}$, parametrized by $u_3$. Equivalently, if $H_{(\xi_1, \xi_2)}$ is an algebraic coset defined by 
\begin{equation*}
M_1=e^{\xi_1},\quad M_2=e^{\xi_2},\quad L_2=e^{v_2(\xi_1, \xi_2)},
\end{equation*}
then $\mathcal{X}\cap H_{(\xi_1, \xi_2)}$ is a $1$-dimensional anomalous subvariety of $\mathcal{X}$. Clearly,  
\begin{equation*}
\bigcup_{(\xi_1, \xi_2)\in \mathbb{C}^2}\Big(\big(u_1=\xi_1,u_2=\xi_2, v_2=v_2(\xi_1, \xi_2)\big)\cap \log\mathcal{X}\Big)=\log\mathcal{X}
\end{equation*}
and so $\mathcal{X}^{oa}=\emptyset$. On the other hand, none of the cusps of M is strongly geometrically isolated from the others. Indeed, from \eqref{19080601}, each $v_i$ depends on variables associated with other cusps (for example, $v_2$ depends on $u_1$, and $v_3$ depends on $u_1$), and hence no cusp is SGI.

Inspired by this, we further refine and generalize \eqref{isolation} as below. 
\begin{definition}\label{19080602}
Let $\mathcal{M}$ be an $n$-cusped hyperbolic $3$-manifold ($n\geq 3$). Suppose $k, l$ are integers such that $0<k<l\leq n$. We say that cusps $1, \dots, k$ are weakly geometrically isolated (WGI) from cusps $k+1,\dots, l$ if each 
\begin{equation*}
v_i(u_1, \dots, u_l, 0, \dots, 0)\quad (1\leq i\leq k)
\end{equation*}
depends only on $u_1, \dots, u_k$, not on $u_{k+1},\dots, u_{l}$. In other words, keeping cusps $l+1,\dots, n$ complete, if cusps $1, \dots, k$ are SGI from cusps $k+1, \dots, l$, then we say cusps $1, \dots, k$ are WGI from cusps $k+1, \dots, l$.
\end{definition}
For instance, if $u_1=0$ in \eqref{19080601}, it is reduced to 
\begin{equation*}
v_2=\frac{1}{2}\sum^{\infty}_{i:even}ia_{2i}u_2^{i-1},\quad v_3=\frac{1}{2}\sum^{\infty}_{i:even}ia_{3i}u_3^{i-1},
\end{equation*}
and therefore the second cusp is WGI from the third cusp in the example. 

Using Definition \ref{19080602}, Theorem \ref{19082701} is simply restated as follows:
\begin{theorem}
Let $\mathcal{M}$ be an $n$-cusped ($n\geq 2$) hyperbolic $3$-manifold having rationally independent cusp shapes and $\mathcal{X}$ be its holonomy variety. If 
\begin{equation*}
\mathcal{X}^{oa}=\emptyset,
\end{equation*} 
then there exist cusps of $\mathcal{M}$ which are WGI from other cusps of $\mathcal{M}$. 
\end{theorem}

Experimental evidence from SnapPy proposes that, if $\mathcal{M}$ is a $2$-cusped hyperbolic $3$-manifold whose cusps are SGI, then some covering spaces of $\mathcal{M}$ may possess this property. We will not pursue this further in the paper, but research along this direction would be interesting.

\subsection{Structure theorem}

The following theorem, due to Bombieri, Masser and Zannier, elucidates the general structure of anomalous subvarieties of any given algebraic variety. As noted earlier in the introduction, this theorem will serve as a cornerstone in establishing our main results in the next two sections.

\begin{theorem}\cite{za}\label{struc}
Let $\mathcal{X}$ be an irreducible variety in $\mathbb{G}^n$ of positive dimension defined over $\mathbb{\overline Q}$. \\
(a) For any irreducible algebraic subgroup $H$ with 
\begin{equation} \label{dim1}
1\leq n- \dim H \leq \dim \mathcal{X},
\end{equation}
the union $\mathcal{Z}_H$ of all subvarieties $\mathcal{Y}$ of $\mathcal{X}$ contained in any algebraic coset of $H$ with
\begin{equation} \label{dim2}
\dim H= n-(1+\dim\mathcal{X})+\dim \mathcal{Y} 
\end{equation}
is a closed subset of $\mathcal{X}$.\\
(b) There is a finite collection $\Psi=\Psi_{\mathcal{X}}$ of such algebraic subgroups $H$ so that every maximal anomalous subvariety $\mathcal{Y}$ of $\mathcal{X}$ is a component of $\mathcal{X} \cap gH$
for some $H$ in $\Psi$ satisfying \eqref{dim1} and \eqref{dim2} and some $g$ in $\mathcal{Z}_H$. Moreover $\mathcal{X}^{oa}$ is obtained from $\mathcal{X}$ by removing the $\mathcal{Z}_H$ of all $H$ in $\Psi$, 
and thus it is Zariski open in $\mathcal{X}$.
\end{theorem} 
Examples explaining the above theorem were already provided in the previous subsection. For $\mathcal{X}=\mathcal{X}_1\times \mathcal{X}_2(\subset\mathbb{G}^{2k}\times \mathbb{G}^{2n-2k})$ given in the proof of Proposition \ref{20041602}, as each $L_i$ ($1\leq i\leq k$) depends only on $M_1, \dots, M_k$ over $\mathcal{X}_1$, if $H$ is an algebraic subgroup defined by 
\begin{equation*}
M_1=\cdots=M_k=L_i=1\quad (1\leq i\leq k), 
\end{equation*}
then one can check $\mathcal{Z}_H=\mathcal{X}$. 
 
If $\mathcal{Z}_H=\mathcal{X}$ (and so $\mathcal{X}^{oa}=\emptyset$), we say $\mathcal{X}$ \emph{is foliated or degenerated by anomalous subvarieties contained in} 
\begin{equation*}
\bigcup_{g\in \mathcal{Z}_H}\mathcal{X}\cap gH
\end{equation*} 
\emph{or algebraic cosets of $H$}. 
\\

\noindent\textbf{Remark.} Given the structure theorem above, one might wonder whether $\mathcal{X}^{oa}\neq \emptyset$ is a common characteristic or not for a given algebraic variety $\mathcal{X}$ in general. Confirming this by any objective or reasonable means, however, appears to be a hard task. Nonetheless, once we narrow our attention down to the class of holonomy varieties of hyperbolic $3$-manifolds, one of our main results, Theorem \ref{19082701}, says that $\mathcal{X}^{oa}\neq \emptyset$ would be the generic behavior, as explained in Section \ref{23081503} (see the discussion after Theorem \ref{20033003}).

\section{Main Result I}
The goal of this section is to prove our first main result, Theorem \ref{19072301}. To this end, we first establish a couple of lemmas in Subsection \ref{260422011} that will be used repeatedly in the proof. Then, as a warm-up, we treat the simplest cases of Theorems \ref{19072301} and \ref{19082701} in Subsection \ref{26042202}. Finally, in the last subsection, we prove Theorem \ref{19072301} in full generality.

\subsection{Preliminary lemma}\label{260422011}
The following lemma is central to the proof of Theorem \ref{19072301}. However, for motivational reasons, we recommend that the reader skim only the statement of the lemma and skip its proof on a first reading. It may be more effective to return to the proof after seeing the importance and applicability of the lemma in the argument.

\begin{lemma}\label{040201} 
Let 
\begin{equation}\label{040205}
\{\bold{v_1},\bold{w_1},\dots, \bold{v_n},\bold{w_n}\}
\end{equation}
be a set of vectors in $\mathbb{Q}^n$ where $n\geq 2$. Suppose, for any subset 
\begin{equation}\label{040206}
\{\bold{u_1},\dots, \bold{u_n}\}
\end{equation}
of \eqref{040205} where $\bold{u_i}=\bold{v_i}$ or $\bold{w_i}$ ($1 \leq i\leq n$), the vectors in \eqref{040206} are linearly dependent. Then there exists $\{i_1,\dots,i_m\}\subsetneq\{1,\dots,n\}$ ($m\geq 1$) such that the dimension of the vector space spanned by 
\begin{equation}\label{040204}
\{\bold{v_{i_1}},\bold{w_{i_1}},\dots,\bold{v_{i_m}},\bold{w_{i_m}}\}
\end{equation}
is at most $m$. 
\end{lemma}
Note that $\bold{v_i}$ or $\bold{w_i}$ could be the zero vector. For instance, if 
\begin{equation*}
\bold{v_i}=\bold{w_i}=\bold{0}, 
\end{equation*}
for some $i$, we get the desired result simply by letting \eqref{040204} be $\{\bold{v_i}, \bold{w_i}\}$. 

The proof of Lemma \ref{040201} relies heavily on the following theorem of R. Rado \cite{rado}.\footnote{In an earlier version of the paper, we claimed the lemma without appealing to Rado’s theorem. However, it turns out that what we actually proved was a special case of Rado’s theorem; thus, by appealing to it, the proof can be significantly streamlined.}

\begin{theorem}[R. Rado]
Let $A_1,\dots,A_n$ be finite subsets of a vector space over ${\mathbb{Q}}$.
Then there exist $\bold{u_i}\in A_i$ $(1\le i\le n)$ such that $\bold{u_1},\dots,\bold{u_n}$ are linearly independent if and only if
\[
\dim_{\mathbb{Q}} \mathrm{span}\bigl(\bigcup_{i\in I} A_i\bigr)\ge |I|
\]
for every $I\subseteq \{1,\dots,n\}$.
\end{theorem}

\begin{proof}[Proof of Lemma \ref{040201}]
For each $i\in\{1,\dots,n\}$, let $A_i:=\{\bold{v_i},\bold{w_i}\}\subset \mathbb{Q}^n.$ By Rado's theorem, there exists a choice
$\bold{u_i}\in A_i$ $(1\le i\le n)$ such that $\bold{u_1},\dots,\bold{u_n}$ are linearly independent if and only if, for every subset $I\subseteq \{1,\dots,n\}$,
\[
\dim_{\mathbb{Q}} \mathrm{span}\bigl(\bigcup_{i\in I} A_i\bigr)\ge |I|.
\]
By hypothesis, since no such independent choice exists, there is a nonempty
subset $I\subseteq \{1,\dots,n\}$ such that
\begin{equation}\label{eq:rank-defect}
\dim_{\mathbb{Q}} \mathrm{span}\bigl(\bigcup_{i\in I} A_i\bigr)<|I|.
\end{equation}
Choose $I$ minimal with respect to inclusion among all subsets satisfying
\eqref{eq:rank-defect}.

If $|I|=1$, say $I=\{i\}$, then \eqref{eq:rank-defect} yields
\[
\dim_{\mathbb{Q}} \mathrm{span}\{\bold{v_i},\bold{w_i}\}<1,
\]
implying $\bold{v_i}=\bold{w_i}=0$. Hence the conclusion is immediate in this case. 

Now suppose $|I|\geq 2$ and let $i\in I$ be arbitrary. By the minimality of $I$, 
\begin{equation}\label{eq:lower-bound}
\dim_{\mathbb{Q}} \mathrm{span}\bigl(\bigcup_{j\in I\setminus\{i\}} A_j\bigr)\ge |I|-1.
\end{equation}
On the other hand, since
\[
\mathrm{span}\bigl(\bigcup_{j\in I\setminus\{i\}} A_j\bigr)
\subseteq
\mathrm{span}\bigl(\bigcup_{j\in I} A_j\bigr)
\]
and
\begin{equation*}
\dim_{\mathbb{Q}} \mathrm{span}\bigl(\bigcup_{j\in I\setminus\{i\}} A_j\bigr)<|I| \quad (\text{by}\;\; \eqref{eq:rank-defect}),  
\end{equation*}
combining it with \eqref{eq:lower-bound}, it follows that 
\begin{equation}\label{eq:exact-rank}
\dim_{\mathbb{Q}} \mathrm{span}\bigl(\bigcup_{j\in I\setminus\{i\}} A_j\bigr)=|I|-1.
\end{equation}

Consequently, we get the desired result by letting 
$I\setminus\{i\}=\{i_1,\dots,i_m\}$. This proves the lemma.
\end{proof}

As an initial application of Lemma \ref{040201}, we prove the following, which will be utilized directly in the proof of Proposition \ref{17062401} in the next subsection.
\begin{lemma}\label{17052301}
Let
\begin{gather}\label{17052401}
\begin{array}{c}
(a_{ij}\;\; b_{ij})_{1\le i,j \le 2} 
\end{array} 
\end{gather}
be an integer matrix of rank $2$, and $\tau_1,\tau_2$ be algebraic numbers such that $1,\tau_1,\tau_2,\tau_1\tau_2$ are linearly independent over $\mathbb{Q}$. 
If the rank of the following matrix  
\begin{equation}\label{22032101}
\begin{array}{c}
(a_{ij}+b_{ij}\tau_j)_{1\le i,j \le 2} 
\end{array} 
\end{equation}
is equal to $1$, then either $a_{i1}=b_{i1}=0$ or $a_{i2}=b_{i2}=0$ ($i=1,2$). 
\end{lemma}
\begin{proof}
Let   
\begin{equation*}
\bold{v_j}:=\left(\begin{array}{c}
a_{1j}\\
a_{2j}
\end{array}
\right)\quad\text{and}\quad 
\bold{w_j}:=\left(\begin{array}{c}
b_{1j}\\
b_{2j}
\end{array}
\right)
\end{equation*}
for $1\le j\le 2$. Since the rank of \eqref{22032101} is $1$, equivalently, we have
\begin{equation*}
\det\left(\begin{array}{cc}
\vert & \vert \\
\bold{u_1} & \bold{u_2}\\
\vert & \vert 
\end{array}\right)=0
\end{equation*}
where $\bold{u_j}=\bold{v_j}$ or $\bold{w_j}$ for each $1\leq j\leq 2$. Thus the rank of either  
\begin{equation}\label{22032103}
\begin{array}{c}
(a_{i1}\;\; b_{i1})_{1\le i \le 2} 
\end{array} 
\quad \text{or}\quad 
\begin{array}{c}
(a_{i2}\;\; b_{i2})_{1\le i \le 2} 
\end{array} 
\end{equation} 
is at most $1$ by Lemma \ref{040201}. Without loss of generality, we assume the first case and $a_{21}=b_{21}=0$ by applying the Gauss elimination if necessary. Again, as the determinant of \eqref{22032101} is $0$, it follows that either $a_{11}+\tau_1b_{11}=0$ or $a_{22}+b_{22}\tau_2=0$, and hence $a_{11}=b_{11}=0$ or $a_{22}=b_{22}=0$ respectively in \eqref{22032101}. If $a_{22}=b_{22}=0$, it contradicts the fact that the rank of \eqref{17052401} is $2$ and so $a_{11}=b_{11}=0$. 

Similarly, if the rank of the second matrix in \eqref{22032103} is at most $1$, then $a_{i2}=b_{i2}=0$ for $i=1,2$.
\end{proof}

\subsection{Codimension $1$}\label{26042202}
We prove Theorems \ref{19072301} and \ref{19082701} under the assumption that the given anomalous subvarieties all have codimension $1$. The proofs in these cases are not only simpler than in the general setting, but also illustrate how the ideas presented in Theorems \ref{potential} and \ref{struc} are applied to obtain the desired results. Thus, they provide a useful overview of the general proof strategy.

First, using Lemma \ref{17052301}, we prove the following result, which can be regarded as a special case of Theorem \ref{19072301}.

\begin{proposition}\label{17062401}
Let $\mathcal{M}$ be an $n$-cusped ($n\geq 2$) hyperbolic $3$-manifold having rationally independent cusp shapes and $\mathcal{X}$ be its holonomy variety. Let $H$ be an irreducible algebraic subgroup of codimension $2$ such that $\mathcal{X}\cap H$ is an anomalous subvariety of $\mathcal{X}$ containing $(1,\dots, 1)$. Then  
\begin{equation}\label{17062402}
\mathcal{X}\cap H=\mathcal{X}\cap (M_j=L_j=1)
\end{equation}
for some $1\leq j\leq n$.
\end{proposition}
\begin{proof}
Let $H$ be defined by
\begin{equation}\label{17052308}
\begin{gathered}
M_1^{a_{i1}}L_1^{b_{i1}}\cdots M_n^{a_{in}}L_n^{b_{in}}=1 \quad (i=1,2).
\end{gathered}
\end{equation}
As remarked in Section \ref{A-poly}, $X\cap H$ is locally biholomorphic to $\log(\mathcal{X}\cap H)$ given by
\begin{equation}\label{17052101}
\begin{gathered}
a_{i1}u_1+b_{i1}(\tau_1u_1+\cdots)+\cdots+a_{in}u_n+b_{in}(\tau_nu_n+\cdots)=0\quad (i=1,2).
\end{gathered}
\end{equation}
Since $\mathcal{X}\cap H$ is an $(n-1)$-dimensional variety, \eqref{17052101} is an $(n-1)$-dimensional complex manifold and thus the rank of  
\begin{equation*}
\begin{array}{c}
(a_{ij}+\tau_jb_{ij})_{1 \le i\le 2,\; 1\le j \le n} 
\end{array} 
\end{equation*}
is equal to $1$. By Lemma \ref{17052301}, for $1\leq j\neq k\leq n $, we have either 
\begin{equation*}
a_{ij}=b_{ij}=0\quad \text{or}\quad a_{ik}=b_{ik}=0\quad (i=1,2).
\end{equation*}
In other words, \eqref{17052308} is reduced to 
\begin{equation*}
\begin{gathered}
M_j^{a_{ij}}L_j^{b_{ij}}=1\quad (i=1,2)
\end{gathered}
\end{equation*}
for some $1\leq j\leq n$. Since $a_{1j}b_{2j}-a_{2j}b_{1j}\neq 0$ and $(1, \dots, 1)\in \mathcal{X}\cap H$, the result follows.  
\end{proof}

Using Proposition \ref{17062401}, we prove a special case of Theorem \ref{19082701}. The result is also viewed as a straightforward generalization of Theorem \ref{20033003}. 

\begin{proposition}\label{17061701}
Let $\mathcal{M}$ and $\mathcal{X}$ be the same as above. Suppose $\mathcal{X}^{oa}=\emptyset$ and, further, $\mathcal{X}$ has infinitely many maximal anomalous subvarieties of dimension $n-1$. Then $\mathcal{M}$ has a cusp which is SGI from the rest.  
\end{proposition}
\begin{proof}
By Theorem \ref{struc}, there exists an irreducible algebraic subgroup $H$ of codimension $2$ such that those anomalous subvarieties are contained in translations of $H$. By Theorem \ref{19072301}, $H$ is $M_j=L_j=1$ for some $1\leq j\leq n$ and, without loss of generality, let us assume $j=1$. If 
\begin{equation*}
\mathcal{X}\cap (M_1=\xi_1,\; L_1=\xi_2)\quad (\xi_1, \xi_2 \in \mathbb{G})
\end{equation*}
is an $(n-1)$-dimensional anomalous subvariety of $\mathcal{X}$, equivalently, 
\begin{equation*}
u_1=\log\xi_1, \; v_1=\log\xi_2
\end{equation*}
is also an $(n-1)$-dimensional analytic subset of $\log\mathcal{X}$. However, this is possible if and only if $v_1$ depends solely on $u_1$. That is, the first cusp is SGI from the rest. 
\end{proof}

\subsection{Proof of Theorem \ref{19072301}}\label{26042203}
We now prove Theorem \ref{19072301}. Throughout this subsection, let $\mathcal{M}$ and $\mathcal{X}$ be as in the statement of the theorem.

In the following proposition, we first establish the theorem under the assumption that the codimension of a given algebraic subgroup is less than or equal to $n$, the dimension of $\mathcal{X}$. This is the essential case, and the general case will then follow almost as a corollary.

\begin{proposition}\label{anomalous3}
Let $H$ be an irreducible algebraic subgroup given by 
\begin{equation*}
\begin{gathered}
M_1^{a_{i1}}L_1^{b_{i1}}\cdots M_n^{a_{in}}L_n^{b_{in}}=1\quad (1\leq i\leq m),
\end{gathered}
\end{equation*}
where $n\geq m$. Then $\mathcal{X}\cap H$ is an anomalous subvariety of $\mathcal{X}$ if and only if the rank of 
\begin{equation}\label{17121403}
\begin{gathered}
\begin{array}{c}
(a_{ij}+\tau_j b_{ij})_{1\le i\le m,\;1\le j\le n} 
\end{array} 
\end{gathered}
\end{equation}
is strictly less than $m$. Moreover, every anomalous subvariety $\mathcal{X}\cap H$ of $\mathcal{X}$ satisfies
\begin{align*}
\mathcal{X}\cap H \subset (M_{j}=L_{j}=1) 
\end{align*}
for some $1\leq j\leq n$. 
\end{proposition}
\begin{proof}
The `only if' direction is clear. Indeed, if the rank of \eqref{17121403} is $m$, by the implicit function theorem, $\dim \big(\log(\mathcal{X}\cap H)\big)=n-m$ and so $\dim (\mathcal{X}\cap H)=n-m$. However, as $\dim H=2n-m$ and $\dim \mathcal{X}=n$ in $\mathbb{G}^{2n}$, if $\mathcal{X}\cap H$ is an anomalous subvariety of $\mathcal{X}$, then the dimension of $\mathcal{X}\cap H$ must be strictly larger than $n-m$, according to \eqref{20031301}. But this is a contradiction.  

Now we prove the `if' direction. For each fixed $n$, we prove the theorem by induction on $m$. Note that the theorem is true for any $n\geq 2$ and $m=2$ by Proposition \ref{17062401}. Now assume $m\geq 3$ and the statement holds for $2, \dots, m-1$. We show that the result is true for $m$ as well. 
 
Since the rank of \eqref{17121403} is less than $m$, the determinant of 
\begin{equation}\label{22031105}
\begin{array}{c}
(a_{ij}+\tau_j b_{ij})_{1 \le i, j \le m}
\end{array}
\end{equation}
is $0$. Since $\mathcal{M}$ has rationally independent cusp shapes by the assumption, if   
\begin{equation*}
\bold{v_j}:=\begin{array}{c}
(a_{ij})_{1\leq i\leq m}
\end{array}, \quad 
\bold{w_j}:=\begin{array}{c}
(b_{ij})_{1\leq i\leq m}
\end{array},
\end{equation*}
then 
\begin{equation*}
\det\left(\begin{array}{ccc}
\vert & &\vert \\
\bold{u_1} & \cdots &\bold{u_m}\\
\vert & &\vert 
\end{array}\right)=0
\end{equation*}
where $\bold{u_j}=\bold{v_j}$ or $\bold{w_j}$ for each $1\leq j\leq m$. By Lemma \ref{040201}, there is 
\begin{equation*}
\{j_1,\dots,j_l\}\subsetneq\{1,\dots,m\}\quad (l<m) 
\end{equation*}
such that the rank of 
\begin{equation*}
\begin{array}{c}
(a_{ij_k}\;\; b_{ij_k})_{1 \le i\le m,\; 1\le k \le l} 
\end{array} 
\end{equation*}
is at most $l$. Let $l$ be the smallest number possessing this property and, without loss of generality, assume $j_k=k$ for $1\leq k\leq l$. Applying Gauss elimination if necessary, we further suppose the coefficient matrix of $H$ and the matrix in \eqref{17121403} are given as  
\[\left(
\begin{array}{c|c}
(a_{ij}\;\; b_{ij})_{1 \le i, j \le l} 
&
(a_{ij}\;\; b_{ij})_{1 \le i \le l,\; l+1 \le j \le n} \\
\hline
0 
&
(a_{ij}\;\; b_{ij})_{l+1 \le i \le m,\; l+1 \le j \le n}
\end{array}
\right)
\]
and
\begin{equation}\label{17071302}
\left(\begin{array}{c|c}
(a_{ij}+\tau_j b_{ij})_{1 \le i, j \le l} 
&
(a_{ij}+\tau_j b_{ij})_{1 \le i \le l, \;l+1 \le j \le n} \\
\hline
0  
&
(a_{ij}+\tau_j b_{ij})_{l+1 \le i \le m, \;l+1 \le j \le n}
\end{array}\right)
\end{equation}
respectively.
\begin{enumerate}
\item If the rank of  
\begin{equation}\label{17072001}
\begin{array}{c}
(a_{ij}+\tau_j b_{ij})_{1 \le i, j \le l}
\end{array}
\end{equation}
is $l$, then the rank of the following submatrix
\begin{equation*}
\begin{array}{c}
(a_{ij}+\tau_j b_{ij})_{l+1 \le i \le m, \;l+1 \le j \le n}
\end{array}
\end{equation*}
of \eqref{17071302} is strictly less than $m-l$ (otherwise, it contradicts the fact that the rank of \eqref{17071302} is strictly less than $m$). If $H'$ is an algebraic subgroup  defined by 
\begin{equation*}
\begin{gathered}
M_{l+1}^{a_{i(l+1)}}L_{l+1}^{b_{i(l+1)}}\cdots M_n^{a_{in}}L_n^{b_{in}}=1\quad (l+1\leq i\leq m),
\end{gathered}
\end{equation*} 
by induction, $\mathcal{X}\cap H'$ is an anomalous subvariety of $\mathcal{X}$ containing $\mathcal{X}\cap H$ and contained in 
\begin{equation*}
\begin{gathered}
M_j=L_j=1
\end{gathered}
\end{equation*}
for some $l+1\leq j\leq m$.

\item Suppose the rank of \eqref{17072001} is strictly less than $l$. By Lemma \ref{040201}, there exists
\begin{equation*}\label{17071402}
\{j_1,\dots,j_{h}\}\subsetneq\{1,\dots,l\}
\end{equation*}
such that the rank of 
\begin{equation*}
\begin{array}{c}
(a_{ij_k}\;\; b_{1j_k})_{1\le i\le l, \; 1\le k\le h}
\end{array} 
\end{equation*}
is strictly less than $h$. But this contradicts the assumption on $l$.  
\end{enumerate}
\end{proof}

Now we complete the proof of Theorem \ref{19072301}, restated simply as follows:\\

\noindent\textbf{Theorem \ref{19072301}.}
\textit{If $H$ is an irreducible algebraic subgroup such that $\mathcal{X}\cap H$ is an anomalous subvariety of $\mathcal{X}$, then  
\begin{equation*}
\mathcal{X}\cap H\subset (M_j=L_j=1)
\end{equation*}
for some $1\leq j\leq n$. }
\begin{proof}
Let $H$ be defined by
\begin{equation}\label{20050401}
\begin{gathered}
M_1^{a_{i1}}L_1^{b_{i1}}\cdots M_n^{a_{in}}L_n^{b_{in}}=1,\quad (1\leq i\leq m).
\end{gathered}
\end{equation}
For $n\geq m$, the theorem was proved in Proposition \ref{anomalous3}, so we assume $m>n$. If $H'$ is an algebraic subgroup defined by the first $n$ equations in \eqref{20050401}, then $\log(\mathcal{X}\cap H')$ is given as
\begin{equation}\label{17072604} 
\begin{gathered}
a_{i1}u_1+b_{i1}(\tau_1u_1+\cdots)+\cdots+a_{in}u_n+b_{in}(\tau_nu_n+\cdots)=0,\quad (1\leq i\leq n)
\end{gathered}
\end{equation}
and the Jacobian of \eqref{17072604} at $(0,\dots,0)$ is
\begin{equation}\label{17072605}
\begin{gathered}
\begin{array}{c}
(a_{ij}+\tau_j b_{ij})_{1\le i,j\le n} 
\end{array} 
\end{gathered}
\end{equation}

If the determinant of \eqref{17072605} is nonzero, by the inverse function theorem, \eqref{17072604} is equivalent to
\begin{equation*}
u_1=\cdots=u_n=0,
\end{equation*}
implying
\begin{equation*}
\mathcal{X}\cap H'=\mathcal{X}\cap H= \mathcal{X}\cap (M_1=\cdots=M_n=1).
\end{equation*}
But this contradicts the fact that the dimension of $\mathcal{X}\cap H$ is positive. 

If the determinant of \eqref{17072605} is zero, by Proposition \ref{anomalous3}, 
\begin{equation*}
\mathcal{X}\cap H'\subset (M_j=L_j=1)
\end{equation*} 
for some $1\leq j\leq n$. This completes the proof of the theorem.
\end{proof}

\section{Main Result II}

In this section, we prove our second main result, Theorem~\ref{19082701}. Before proceeding with the proof, some preliminaries are needed. In Subsection~\ref{25062701}, we introduce a definition and prove a proposition that will be used in the proof of the theorem. We then use these to present the proof of Theorem~\ref{19082701} in the following subsection.

\subsection{Preliminary proposition}\label{25062701}

Let us begin by refining the definition of an anomalous subvariety given in Definition \ref{20040802}: 
\begin{definition}
Let $\mathcal{X}$ be an irreducible variety in $\mathbb{G}^n$ and $b\geq 0$ be an integer. We say an irreducible subvariety $\mathcal{Y}$ of $\mathcal{X}$ is $b$-anomalous if it lies in an algebraic coset $H (\subset \mathbb{G}^n)$ satisfying
\begin{equation*}
\dim \mathcal{Y}=\dim H+\dim \mathcal{X} -n+b.
\end{equation*}
\end{definition}
The above definition first appeared in \cite{BMZ1}. In the original definition, $b$ is assumed to be positive, but $b=0$ is allowed in our case. (This is for the sake of convenience in the proof of the proposition below.) Note that $0$-anomalous subvarieties are not anomalous in the sense of Definition \ref{20040802}. 

For instance, if $\mathcal{X}$ is the holonomy variety of an $n$-cusped hyperbolic $3$-manifold, then 
\begin{equation*}
\mathcal{X}\cap (M_{j_1}=L_{j_1}=\cdots=M_{j_m}=L_{j_m}=1)\quad (0<m<n)
\end{equation*}
is an $m$-anomalous subvariety of $\mathcal{X}$. 

In general, let $H^{(m)}$ be an algebraic subgroup of codimension $m$ defined by 
\begin{equation}\label{22031101}
M_1^{a_{i1}}L_1^{b_{i1}}\cdots M_m^{a_{im}}L_m^{b_{im}}=1\quad (1\leq i\leq m). 
\end{equation}
If the rank of the Jacobian matrix associated to $\log (\mathcal{X}\cap H^{(m)})$ is $m$, then 
\begin{equation*}
\mathcal{X}\cap H^{(m)}=\mathcal{X}\cap (M_1=\cdots =M_m=1)
\end{equation*} 
by Proposition \ref{anomalous2}, that is, $\mathcal{X}\cap H^{(m)}$ is a $0$-anomalous variety of $\mathcal{X}$. Further, if we add 
\begin{equation}\label{22031102}
M_1^{a_{i1}}L_1^{b_{i1}}\cdots M_m^{a_{im}}L_m^{b_{im}}=1\quad (m+1\leq i\leq m+b)
\end{equation}
on top of \eqref{22031101} and define $H^{(m+b)}$ as an algebraic subgroup of codimension $m+b$ by \eqref{22031101}-\eqref{22031102}, then
\begin{equation*}
\mathcal{X}\cap H^{(m)}=\mathcal{X}\cap H^{(m+b)}=\mathcal{X}\cap (M_1=\cdots =M_m=1)
\end{equation*}
and thus $\mathcal{X}\cap H^{(m+b)}$ is a $b$-anomalous subvariety of $\mathcal{X}$. In the following, we show that this is always the case, that is, every $b$-anomalous subvariety of $\mathcal{X}$ arises in this manner.  \\

\noindent\textbf{Convention.} To simplify notation, we write $M_{[i,j]}=1$ to denote $M_i = \cdots = M_j = 1$.

\begin{proposition}\label{20040601}
Let $\mathcal{M}$ and $\mathcal{X}$ be the same as in Theorem \ref{19082701}. Let $H$ be an irreducible algebraic subgroup such that $\mathcal{X}\cap H$ is a $b$-anomalous, but not $(b+1)$-anomalous, subvariety of $\mathcal{X}$ and suppose that
\begin{equation*}
\begin{gathered}
\mathcal{X}\cap H\subset (M_{[1,m]}=1),\quad \mathcal{X}\cap H\not\subset (M_{j}=1) 
\end{gathered}
\end{equation*}
for $j\in \{m+1, \dots, n\}$. Then $b\leq m$ and there exists algebraic subgroup $H^{(m+b)}$ and $H^{(m)}$ of codimension $m+b$ and $m$ respectively such that 
\begin{enumerate}
\item $H\subset H^{(m+b)}\subset H^{(m)}$,\;\; $\mathcal{X}\cap H^{(m+b)}=\mathcal{X}\cap H^{(m)}=\mathcal{X}\cap (M_{[1,m]}=1)$; 
\item $H^{(m+b)}$ is defined by equations involving only the variables $M_j, L_j$ for $1\leq j\leq m$, i.e., $H^{(m+b)}$ is defined by equations of the following form:
\begin{equation*}
M_1^{a_{i1}}L_1^{b_{i1}}\cdots M_m^{a_{im}}L_m^{b_{im}}=1, \quad (1\leq i\leq m+b).
\end{equation*}
\end{enumerate}
\end{proposition}

To prove the proposition - or more precisely, to construct $H^{(m+b)}$ and $H^{(m)}$ as stated - we use induction on $n$ and $m$, along with a simple trick involving a canonical projection. To elaborate on this, for $\mathcal{X}$ and $H$ are given as above (with $n\geq 2$), let\footnote{If $H$ is defined by
\begin{equation*}
M_1^{a_{i1}}L_1^{b_{i1}}\cdots M_n^{a_{in}}L_n^{b_{in}}=1,\quad (1\leq i\leq h), 
\end{equation*}
then $H_1$ is simply defined by
\begin{equation*}
M_2^{a_{i2}}L_2^{b_{i2}}\cdots M_n^{a_{in}}L_n^{b_{in}}=1,\quad (1\leq i\leq h). 
\end{equation*}}  
\begin{equation}\label{20040603}
\mathcal{X}_1:=\mathcal{X}\cap (M_1=L_1=1)\quad \text{and}\quad H_1:=H\cap (M_1=L_1=1).   
\end{equation}
Note that $H_1$ satisfies $\dim H-2\leq \dim H_1 \leq \dim H$, and if $\dim H-\dim H_1$ is either $2$ or $1$, then $H$ is contained in either 
\begin{equation}\label{25062401}
M_1=L_1=1\quad \text{or}\quad M_1^cL_1^d=1
\end{equation}
respectively, for some $c,d\in \mathbb{Z}$. We further suppose $\mathcal{X}_{1}$ and $H_{1}$ are in $\mathbb{G}^{2(n-1)}$ via the projection 
\begin{equation*}
\text{Pr}\;:(M_1, L_1, \dots, M_n, L_n)\longrightarrow (M_2, L_2, \dots, M_n, L_n)
\end{equation*}
and regard $\mathcal{X}_{1}$ as the holonomy variety of an $(n-1)$-cusped hyperbolic $3$-manifold. If $\mathcal{X}_1\cap H_1$ is an anomalous subvariety of $\mathcal{X}_1$ (in $\mathbb{G}^{2n-2}$), by the induction hypothesis, there will be an algebraic subgroup containing $H_1$, which can be lifted to an algebraic subgroup containing $H$. Combining this with the one already given in \eqref{25062401}, we will finally be able to construct the desired one $H^{(m+b)}$ (as well as $H^{(m)}$), as stated in the proposition. Of course, a more careful case-by-case analysis of the possible values of $\dim H$, $\dim(\mathcal{X} \cap H)$, and $\dim H_1$ will be required in the actual proof below.\\

\noindent\textbf{Remark. }If $H^{(m)}$ is given as in the statement of Proposition~\ref{20040601}, since $\operatorname{codim} H^{(m)} = m$ and $\dim(\mathcal{X} \cap H^{(m)}) = n - m$, the rank of the Jacobian of $\log(\mathcal{X} \cap H^{(m)})$ is $m$ by Proposition~\ref{anomalous3}. Moreover, since $H^{(m+b)} \subset H^{(m)}$ and $H^{(m+b)}$ is defined by equations involving only $M_j, L_j$ for $1 \leq j \leq m$, the rank of the Jacobian of $H^{(m+b)}$ is also $m$. We will use this fact in the proof below.\\

\begin{proof}[Proof of Proposition \ref{20040601}]
As already mentioned, we proceed by induction on $n$ and $m$. The claim is clearly true when $n = m = 1$.

\begin{enumerate}
\item Suppose $n\geq 2$ and $m=1$. It is enough to show either 
\begin{equation*}
b=1,\quad H\subset (M_1=L_1=1) \quad \text{or} \quad b=0, \quad H\subset (M_1^cL_1^d=1)
\end{equation*}
for some $c,d\in \mathbb{Z}$. Let $\mathcal{X}_{1}, H_{1}(\in \mathbb{G}^{2n-2})$ be as given in \eqref{20040603} and set
$$a:=\dim H-\dim H_{1}\;\; (\in \{0,1,2\}).$$ 
Then 
\begin{equation*}
\begin{gathered}
\dim \mathcal{X}\cap H=\dim \mathcal{X}_{1}\cap H_{1}=\dim\mathcal{X}+\dim H-2n+b, 
\end{gathered}
\end{equation*}
which is  
\begin{equation}\label{20040701}
\begin{gathered}
(\dim\mathcal{X}_{1}+1)+(\dim H_{1}+a)-2n+b=\dim \mathcal{X}_{1}+\dim H_{1}-2(n-1)+a+b-1. 
\end{gathered}
\end{equation}
\begin{enumerate}
\item If $b\geq 2$, then 
\begin{equation*}
\dim \mathcal{X}_{1}\cap H_{1}\geq \dim \mathcal{X}_{1}+\dim H_{1}-2(n-1)+1 
\end{equation*}
and so $\mathcal{X}_{1}\cap H_{1}$ is an anomalous subvariety of $\mathcal{X}_{1}$ (in $\mathbb{G}^{2(n-1)}$). By Theorem \ref{19072301}, there exists some $j$ ($2\leq j\leq n$) such that 
\begin{equation*}
\mathcal{X}_{1}\cap H_{1}\subset (M_j=L_j=1), 
\end{equation*}
which contradicts the assumption on $m$. 

\item For $b=a=1$, one gets the same contradiction as above. For $b=1$ and $a=0$, $\dim H=\dim H_{1}$, which implies $H\subset (M_1=L_1=1)$. By letting $H^{(2)}$ be $M_1=L_1=1$, the result follows.
 
\item If $b=0$, then 
\begin{equation*}
\dim \mathcal{X}_{1}\cap H_{1}=\dim \mathcal{X}_{1}+\dim H_{1}-2(n-1)+a-1.
\end{equation*}
If $a=0$, then this contradicts a well-known result in intersection theory (c.f. Proposition 3.28 in \cite{mum}). If $a=2$, then $\mathcal{X}_{1}\cap H_{1}$ is an anomalous subvariety of $\mathcal{X}_{1}$, again contradicting the assumption on $m$. If $a=1$, i.e., $\dim H=1+\dim H_{1}$, then $H\subset (M_1^cL_1^d=1)$ for some $c,d\in \mathbb{Z}$ by the definition of $H_{1}$. Letting $H^{(1)}$ be defined by $M_1^cL_1^d=1$, the claim follows. 
\end{enumerate}

\item Now suppose $n\geq m\geq 2$ and the claim holds for any $\mathcal{X}$, $H$ satisfying either $\dim \mathcal{X}<n$ or  
\begin{equation*}
\begin{gathered}
\dim \mathcal{X}=n, \quad \mathcal{X}\cap H\subset (M_{j_1}=\cdots=M_{j_l}=1), \quad \mathcal{X}\cap H\not\subset (M_{j}=1) 
\end{gathered}
\end{equation*}
where $l<m$ and $j\in \{1, \dots, n\}\backslash\{j_1, \dots, j_l\}$. 

Let $\mathcal{X}_{1},H_{1}$ and $a$ be the same as above.\footnote{That is, $\mathcal{X}_{1}\cap H_{1}$ is an $(a+b-1)$-anomalous subvariety of $\mathcal{X}_{1}$ in $\mathbb{G}^{2n-2}$.} Since $\dim \mathcal{X}_{1}=n-1$, by the induction hypothesis, 
\begin{equation*}
b+a-1\leq m-1 \Longrightarrow b\leq m
\end{equation*}
and there exists an algebraic subgroup $H^{(m+a+b-2)}_{1}$ of codimension $(m-1)+(a+b-1)$ satisfying 
\begin{equation}\label{21081301}
H_{1}\subset H^{(m+a+b-2)}_{1}, \quad \mathcal{X}_{1}\cap H^{(m+a+b-2)}_{1}=\mathcal{X}_{1}\cap (M_{[2, m]}=1).
\end{equation}
By the definition of $H_{1}$, one further obtains an algebraic subgroup $H^{(m+a+b-2)}$ in $\mathbb{G}^{2n}$ containing $H$, which satisfies
\begin{equation}\label{21081601}
\big(\mathcal{X}\cap H^{(m+a+b-2)} \cap (M_1=1)\big)=\big(\mathcal{X}\cap (M_{[1, m]}=1)\big).
\end{equation}
\begin{enumerate}
\item If $a=1$, it means $H\subset (M_1^cL_1^d=1)$ for some $c,d\in \mathbb{Z}$ and so $H^{(m+b-1)}\cap(M_1^cL_1^d=1)$ is an algebraic subgroup of codimension $m+b$ containing $H$. Since 
\begin{equation*}
\mathcal{X}\cap (M_1^cL_1^d=1)\subset (M_1=1),
\end{equation*}
we get the desired result by letting $H^{(m+b)}:=H^{(m+b-1)}\cap (M_1^cL_1^d=1)$. The existence of $H^{(m)}$ is immediate from the induction, while the second claim concerning $H^{(m+b)}$ is also clear from its construction. 

\item If $a=0$, then $H\subset (M_1=L_1=1)$. Similar to the previous case, the conclusion follows by setting $H^{(m+b)}:=H^{(m+b-2)}\cap (M_1=L_1=1)$. 

\item If $a = 2$, then $H^{(m+b)}$ itself is an algebraic subgroup of codimension \( m+b \) containing \( H \). To complete the proof, it suffices to show the existence of an algebraic subgroup \( H^{(m)} \) that satisfies the required conditions stated in the proposition.

Let \( H(m) \) be any algebraic subgroup of codimension \( m \) containing \( H^{(m+b)} \). If \( \dim(\mathcal{X} \cap H(m)) = n - m \), then the rank of the Jacobian of \( \log(\mathcal{X} \cap H(m)) \) is \( m \) by Proposition~\ref{anomalous3}, and so
\begin{equation}\label{25070201}
\mathcal{X} \cap H(m) = \mathcal{X} \cap (M_{[1,m]} = 1)
\end{equation}
by Proposition \ref{anomalous2}, implying \( H(m)\) itself is exactly the desired one. Otherwise, \( \mathcal{X} \cap H(m) \) is an anomalous subvariety of \( \mathcal{X} \), and hence by Proposition~\ref{anomalous3},
\begin{equation*}
\mathcal{X} \cap H(m) \subset (M_{j_1} = \cdots = M_{j_l} = 1)
\end{equation*}
for some indices \( 1 \leq j_k \leq m \) with \( 1 \leq k \leq l < m \). Without loss of generality, we assume \( l \) is maximal and \( \{j_1, \dots, j_l\} = \{m - l + 1, \dots, m\} \). By induction, there exists an algebraic subgroup \( H^{(l)} \) of codimension $l$ containing \( H(m) \) and
\begin{equation}\label{25070303}
\mathcal{X} \cap H^{(l)} = \mathcal{X} \cap (M_{[m - l + 1, m]} = 1).
\end{equation}
Now, let 
\begin{equation*}
\mathcal{X}_l:=\mathcal{X}\cap (M_{[m-l+1, m]}=1)\; \text{and}\; H_l:=H\cap (M_{[m-l+1, m]}=L_{[m-l+1, m]}=1)
\end{equation*}
and consider both $\mathcal{X}_l$ and $H_l$ as subvarieties in $\mathbb{G}^{2n-2l}$. Since 
\begin{equation*}
\mathcal{X}_l \cap H_l \subset (M_{[1,m-l]}=1),  
\end{equation*}
we again apply induction to obtain $H_l^{(m-l)}$ of codimension $m-l$, which contains $H_l$ and satisfies 
\begin{equation}\label{25070302}
\mathcal{X}_l\cap H_l^{(m-l)}=\mathcal{X}_l\cap (M_{[1,m-l]}=1). 
\end{equation}
By the definition of $H_l$, $H_l^{(m-l)}$ can be further lifted to an algebraic subgroup $H^{(m-l)}$ (of the same codimension) containing $H$ such that\footnote{Note that, since $H^{(m-l)}$ is defined by equations in the variables $M_i, L_i$ for $1\leq i\leq m$, it contains $H^{(m+b)}$ as well.}
\begin{equation}\label{25070301}
H^{(m-l)}\cap (M_{[m-l+1, m]}=L_{[m-l+1, m]}=1)=H_l^{(m-l)}.
\end{equation} 
Finally, define $H^{(m)}:=H^{(l)}\cap H^{(m-l)}$, which then satisfies the desired conditions by \eqref{25070303}-\eqref{25070301}. 
\end{enumerate}
\end{enumerate}
\end{proof}

\subsection{Proof of Theorem \ref{19082701}}\label{20041401}
In this subsection, we prove our second main result, Theorem \ref{19082701}, by breaking it down into several distinct special cases.

First, if $\mathcal{X}^{oa}=\emptyset$, by Theorem \ref{struc}, there exists an irreducible algebraic subgroup $H$ such that $\mathcal{X}\cap H$ is a $1$-anomalous subvariety of $\mathcal{X}$ and $\mathcal{X}=\mathcal{Z}_H$, i.e., $\mathcal{X}$ is foliated by maximal anomalous subvarieties contained in 
\begin{equation}\label{19123001}
\bigcup_{g\in \mathcal{Z}_H}\mathcal{X}\cap gH.
\end{equation}
Let $m$ $(\geq 1)$ be the largest number such that 
\begin{equation}\label{21081315}
\begin{aligned}
\mathcal{X}\cap H&\subset (M_{j_1}=\cdots=M_{j_m}=1)\;\;\text{and}\;\;\mathcal{X}\cap H&\not\subset (M_{j}=1) 
\end{aligned}
\end{equation}
for $j\in \{1, \dots, n\}\backslash\{j_1, \dots, j_m\}$. Without loss of generality, we assume $j_k=k$ for $1\leq k\leq m$. By Proposition \ref{20040601}, $H$ is contained in an algebraic subgroup $H^{(m+1)}$ (resp. $H^{(m)}$) of codimension $m+1$ (resp. $m$), defined by equations of the following form
\begin{equation}\label{21072903}
M_1^{a_{i1}}L_1^{b_{i1}}\cdots M_m^{a_{im}}L_m^{b_{im}}=1,\quad (0\leq i\leq m)\quad (\text{resp. }(1\leq i\leq m)) 
\end{equation}
and satisfying
\begin{equation}\label{21081212}
\mathcal{X}\cap H^{(m)}=\mathcal{X}\cap H^{(m+1)}=\mathcal{X}\cap (M_1=\cdots =M_m=1).
\end{equation} 
Thus if  
\begin{equation}\label{20050403}
\dim (\mathcal{X}\cap H)=n-m-l,\quad \text{codim } H=m+l+1
\end{equation} 
for some $l\geq0$, one may assume $H$ is defined by \eqref{21072903}, together with additional constraints of the form
\begin{equation}\label{20042901}
\begin{gathered}
M_1^{a'_{i1}}L_1^{b'_{i1}}\cdots  M_n^{a'_{in}}L_n^{b'_{in}}=1,\quad (1\leq i\leq l).
\end{gathered}
\end{equation}

We first claim 
\begin{lemma}\label{20040906}
Having the same notation and assumptions as above, there exists an analytic function $\Theta(s_1, \dots, s_m, t_{1},\dots, t_{l})$ such that 
\begin{equation}\label{19082903}
a_{01}u_1+b_{01}v_1+\cdots +a_{0m}u_m+b_{0m}v_m=\Theta(s_1, \dots, s_m, t_{1},\dots, t_{l})
\end{equation}
where 
\begin{equation}\label{21072901}
\begin{aligned}
s_i&=a_{i1}u_1+b_{i1}v_1+\cdots +a_{im}u_m+b_{im}v_m\quad (1\leq i\leq m),\\
t_i&=a'_{i1}u_1+b'_{i1}v_1+\cdots +a'_{in}u_n+b'_{in}v_n \quad (1\leq i\leq l).
\end{aligned}
\end{equation}
\end{lemma}

\begin{proof}
Since $\mathcal{X}$ is foliated by anomalous subvarieties in \eqref{19123001}, equivalently, $\log\mathcal{X}$ is foliated by elements in $\bigcup_{g\in \mathcal{Z}_H}\log(\mathcal{X}\cap gH)$. As each $\log(\mathcal{X}\cap gH)$ is defined by equations of the following types
\begin{equation}\label{19081601}
\begin{aligned}
\zeta_i&=a_{i1}u_1+b_{i1}v_1+\cdots +a_{im}u_m+b_{im}v_m\quad (0\leq i\leq m),\\
\zeta'_i&=a'_{i1}u_1+b'_{i1}v_1+\cdots+a'_{in}u_n+b'_{in}v_n\quad (1\leq i\leq l).
\end{aligned}
\end{equation}
where $\zeta_i, \zeta'_i\in \mathbb{C}$, if we set
\begin{equation*}
\begin{gathered}
T:=\{(\zeta_0, \dots, \zeta'_l)\in \mathbb{C}^{m+l+1}:\;\eqref{19081601} \text{ is a complex manifold of dimension }n-m-l \},
\end{gathered}
\end{equation*} 
then 
\begin{equation*}
\dim T=\dim(\log \mathcal{X})-\dim\big(\log(\mathcal{X}\cap H)\big)=n-(n-m-l)=m+l.
\end{equation*}
Thus $T$ is a hypersurface in $\mathbb{C}^{m+l+1}$ and this implies there exists $\Theta$ satisfying \eqref{19082903}.  
\end{proof}
Later in Section \ref{20040907}, it will be shown that $l=0$. At the moment, let us assume $l=0$ and consider the following two subcases:\\

\hspace{0.5cm} $\clubsuit$ $l=0$ and there is no $H'$ such that $H\subsetneq H'$ and $\mathcal{X}\cap H'$ is an anomalous subvariety of $\mathcal{X}$;\\

\hspace{0.5cm} $\spadesuit$ $l=0$ and there exists $H'$ such that $H\subsetneq H'$ and $\mathcal{X}\cap H'$ is an anomalous subvariety of $\mathcal{X}$.\\

In the first case, we show that cusps $1, \dots, m$ are SGI from the rest and, in the second, find a proper subset of cusps $1, \dots, m$ which are WGI from cusps $m+1, \dots, n$.

\subsubsection{$\clubsuit$ $\Longrightarrow$ SGI}\label{25070207}
In this subsubsection, we prove the following theorem:
\begin{theorem}\label{21072807}
Let $\mathcal{M}$, $\mathcal{X}$ and $H$ be the same as above. If $H$ satisfies the assumption in $\clubsuit$, then cusps $1, \dots, m$ of $\mathcal{M}$ are SGI from the rest.
\end{theorem}

The proof is by contradiction. If cusps $1, \dots, m$ are not SGI from the rest, there exists $v_i$ ($1\leq i\leq m$) having a term divisible by some $u_j$ ($m+1\leq j\leq n$). We find such a term of the lowest degree, compare two coefficients of the term in \eqref{19082903} and get an equality involving $a_{ij}, b_{ij}, \tau_j$ ($0\leq i,j\leq m$). By carefully analyzing and manipulating the equality, we will show the existence of $H'$ containing $H$ properly such that $\mathcal{X} \cap H'$ is an anomalous subvariety of $\mathcal{X}$, thereby contradicting the assumption made in $\clubsuit$.
 
\begin{proof}
Let
\begin{equation*}
\Phi(u_1, \dots, u_n)=\sum^{\infty}_{(\alpha_1, \dots, \alpha_n)\in(\mathbb{Z}^+)^n} c_{\alpha_1, \dots, \alpha_n}u_1^{\alpha_1}\cdots u_n^{\alpha_n}
\end{equation*}
be the Neumann-Zagier potential function of $\mathcal{M}$ and $S$ be the set of all $u_1^{\alpha_1}\cdots u_n^{\alpha_n}$ satisfying 
\begin{itemize}
\item $c_{\alpha_1, \dots, \alpha_n}\neq 0$ with $(\alpha_1, \dots, \alpha_{m})\neq (0, \dots, 0)$ and $(\alpha_{m+1}, \dots, \alpha_{n})\neq (0, \dots, 0)$;
\item $\alpha_1+\cdots +\alpha_{m}$ is minimal.
\end{itemize}
Note that $S=\emptyset$ if and only if cusps $1, \dots, m$ are SGI from cusps $m+1, \dots, n$. So assume $S\neq \emptyset$ and $u_1^{\alpha_1}\cdots u_n^{\alpha_n}\in S$. Without loss of generality, it is further supposed that $\alpha_1\neq 0$ and 
\begin{equation*}
\bold{u}:=\frac{1}{2}\alpha_1c_{\alpha_1, \dots, \alpha_n}u_1^{\alpha_1-1}\cdots u_n^{\alpha_n}.
\end{equation*} 
By Lemma \ref{20040906}, there exists an analytic function  
\begin{equation*}
\Theta(s_1, \dots, s_m):=e_1s_1+\cdots e_{m}s_{m}+\text{higher degrees}, 
\end{equation*}
satisfying \eqref{19082903}. Comparing the coefficients of $\bold{u}$ on both sides of \eqref{19082903}, we get
\begin{equation}\label{21072801}
b_{01}=\left( \begin{array}{c}
e_1   \\
\vdots   \\
e_{m}  
\end{array} \right)^{T}
\left( \begin{array}{c}
b_{11}   \\
\vdots   \\
b_{m1}  
\end{array} \right).
\end{equation}
If $b_{01}\neq 0$, then $b_{i1}\neq 0$ for some $i$ ($1\leq i\leq n$) by \eqref{21072801}. Hence, by applying Gauss elimination if necessary, we suppose $b_{01}=0$ in \eqref{21072903}, \eqref{19082903} and \eqref{21072801}. 

Comparing the coefficients of the linear terms of the both sides in \eqref{19082903}, it follows that
\begin{equation*}
\begin{gathered}
\left( \begin{array}{c}
a_{01}  \\
a_{02}+b_{02}\tau_2 \\
\vdots   \\
a_{0m}+b_{0m}\tau_m 
\end{array} \right)
=A^T\left( \begin{array}{c}
e_1   \\
\vdots   \\
e_{m}  
\end{array} \right)\Longrightarrow \left( \begin{array}{c}
e_1   \\
\vdots   \\
e_{m}  
\end{array} \right)
=A^{-T}
\left( \begin{array}{c}
a_{01}  \\
a_{02}+b_{02}\tau_2   \\
\vdots   \\
a_{0m}+b_{0m}\tau_m  
\end{array} \right)
\end{gathered}
\end{equation*}
where $A:=\begin{array}{c}
(a_{ij}+b_{ij}\tau_j)_{1\leq i,j\leq m}
\end{array}$.\footnote{Note that $A$ is the Jacobian of $\log (\mathcal{X}\cap H^{(m)})$ and it is invertible by the assumptions on $H^{(m)}$ and Proposition \ref{anomalous2}.} 
Combining it with \eqref{21072801}, we further get
\begin{equation}\label{19081610}
\begin{gathered}
0(=b_{01})=\left( \begin{array}{c}
e_1   \\
\vdots   \\
e_{m}  
\end{array} \right)^{T}
\left( \begin{array}{c}
b_{11}   \\
\vdots   \\
b_{m1}  
\end{array} \right)
=\left( \begin{array}{c}
a_{01}   \\
a_{02}+b_{02}\tau_2   \\
\vdots   \\
a_{0m}+b_{0m}\tau_m 
\end{array} \right)^{T}A^{-1}
\left( \begin{array}{c}
b_{11}   \\
\vdots   \\
b_{m1}  
\end{array} \right).
\end{gathered}
\end{equation}

We show the above equality \eqref{19081610} contradicts the condition in $\clubsuit$. Let 
\begin{equation*}
\bold{v_i}:=\begin{array}{ccc}
(a_{i1}+b_{i1}\tau_1   & \hdots & a_{im}+b_{im}\tau_{m}) 
\end{array} \quad (0\leq i\leq m).
\end{equation*}
By the inverse matrix formula, \eqref{19081610} is equivalent to
\begin{equation}\label{19080501}
\begin{gathered}
\small\small
-b_{11}\det
\left( \begin{array}{c}
\bold{v_0}\\
\bold{v_2}\\
\bold{v_3}\\
 \vdots\\
\bold{v_{m}}
\end{array} \right)
+\cdots+b_{m1}(-1)^m\det
\left( \begin{array}{c}
\bold{v_1}\\
\bold{v_2}\\
 \vdots\\
\bold{v_{m-1}}\\
\bold{v_{0}}
\end{array} \right)
=\sum_{i=1}^{m} b_{i1}(-1)^{i}\det
\left( \begin{array}{c}
\bold{v_1}\\
 \vdots\\
\bold{\hat{v_i}}\\
 \vdots \\
\bold{v_{m}}
\end{array} \right)=0
\end{gathered}
\end{equation}
where $\bold{\hat{v_i}}:=\bold{v_0}$ for each $i$. We claim
\begin{claim}\label{19123102}
$b_{i1}\neq 0$ for some $1\leq i\leq m$. 
\end{claim}
\begin{proof}[Proof of Claim \ref{19123102}]
On the contrary, suppose $b_{i1}=0$ for all $i$ ($1\leq i\leq m$). If $a_{i1}=0$ for all $0\leq i\leq m$, it contradicts our initial assumption that the rank of the Jacobian of $\log (\mathcal{X}\cap H^{(n)})$ is $m$ (see the remark below Proposition \ref{anomalous3}). Thus $a_{i1}\neq 0$ for some $i$ and, without loss of generality, we take $a_{m1}\neq 0$. Applying Gauss elimination if necessary, and further setting 
\begin{equation*}
a_{01}=\cdots=a_{(m-1)1}=0, 
\end{equation*}
the rank of the following Jacobian of $\log(\mathcal{X}\cap H)$,   
\begin{equation}\label{19081402}
\begin{gathered}
\left(\begin{array}{c|c}
0 & (a_{ij} + b_{ij} \tau_j)_{0\leq i\leq m-1, \; 2\leq j\leq m} \\
\hline
a_{m1} & (a_{mj} + b_{mj} \tau_j)_{2\leq j\leq m}
\end{array}
\right),
\end{gathered}
\end{equation}
is $m$ by the assumption. Therefore, the rank of the following submatrix  
\begin{equation}\label{21072803}
\begin{gathered}
\begin{array}{c}
(a_{ij}+b_{ij}\tau_j)_{0\leq i\leq m-1, \;2\leq j\leq m}
\end{array}
\end{gathered}
\end{equation}
of \eqref{19081402} is $m-1$. If $H'$ is an algebraic subgroup such that \eqref{21072803} is the Jacobian of $\log (\mathcal{X}\cap H')$, then $\mathcal{X}\cap H'$ is an anomalous subvariety of $\mathcal{X}$ by Proposition \ref{anomalous3}. As $H\subsetneq H'$, it contradicts the condition on $H$ made in $\clubsuit$. In conclusion, $b_{i1}\neq 0$ for some $1\leq i\leq m$.   
\end{proof}
Without loss of generality, set
\[ \begin{cases} 
b_{i1}\neq 0\quad (1\leq i\leq h),\\
b_{i1}=0\quad (h+1\leq i\leq m).
 \end{cases} \]
By elementary properties of determinants, \eqref{19080501} is equivalent to 
\begin{equation}\label{21072805}
\begin{gathered}
\det\left( \begin{array}{c}
b_{11}\bold{v_0}\\
\frac{b_{21}}{b_{11}}\bold{v_1}+\bold{v_2}\\
\vdots \\  
\frac{b_{h1}}{b_{(h-1)1}}\bold{v_{h-1}}+\bold{v_{h}}\\
\bold{v_{h+1}}\\
 \vdots \\
\bold{v_{m}}
\end{array} \right)=
\det\left( \begin{array}{c}
b_{11}\bold{v_0}\\
b_{21}\bold{v_1}+b_{11}\bold{v_2}\\
\vdots \\  
b_{h1}\bold{v_{h-1}}+b_{(h-1)1}\bold{v_{h}}\\
\bold{v_{h+1}}\\
 \vdots \\
\bold{v_{m}}
\end{array} \right)=0.
\end{gathered}
\end{equation}
If $H'$ is an algebraic subgroup such that the Jacobian of $\log(\mathcal{X}\cap H'$) is as given in \eqref{21072805}, then clearly $H'$ is an algebraic subgroup satisfying $H\subsetneq H'$ and $\mathcal{X}\cap H'$ is an anomalous subvariety of $\mathcal{X}$ by Proposition \ref{anomalous3}. However the existence of $H'$  contradicts the assumption made in $\clubsuit$. This completes the proof of Theorem \ref{21072807}. 
\end{proof}

\subsubsection{$\spadesuit$ $\Longrightarrow$ WGI}

Now we consider the second case $\spadesuit$. 

Let $H'$ be an algebraic subgroup such that $H\subsetneq H'$ and $\mathcal{X}\cap H'$ is an anomalous subvariety of $\mathcal{X}$. We further suppose $H'$ is the largest algebraic subgroup satisfying this property. That is, there is no algebraic subgroup $H''$ containing $H'$ properly and $\mathcal{X}\cap H''$ is an anomalous subvariety of $\mathcal{X}$. By the assumption,  
\begin{equation*}
\mathcal{X}\cap H'=\mathcal{X}\cap (M_{j_1}=\cdots=M_{j_h}=1)
\end{equation*}
for some $\{j_1, \dots, j_h\}\subset \{1, \dots, m\}$ and, without loss of generality, we set $j_k=k$ for $1\leq k\leq h$. By Proposition \ref{20040601}, $H'$ is defined by the following types of equations
\begin{equation*}
\begin{gathered}
M_1^{a_{i1}}L_1^{b_{i1}}\cdots M_{h}^{a_{ih}}L_{h}^{b_{ih}}=1\quad (0\leq i\leq h)
\end{gathered}
\end{equation*}
and so $H$ is given by 
 \begin{equation*}
\begin{aligned}
M_1^{a_{i1}}L_1^{b_{i1}}\cdots M_{h}^{a_{ih}}L_{h}^{b_{ih}}&=1\quad (0\leq i\leq h),\\
M_1^{a_{i1}}L_1^{b_{i1}}\cdots M_m^{a_{im}}L_m^{b_{im}}&=1\quad (h+1\leq i\leq m).
\end{aligned}
\end{equation*}
\begin{theorem}\label{21080805}
Let $\mathcal{M}$, $\mathcal{X}$ and $H$ be the same as in Theorem \ref{21072807}. Suppose $H$ satisfies the assumption in $\spadesuit$ and $H'$ is an algebraic subgroup containing $H$ as given above. Then cusps $1, \dots, h$ of $\mathcal{M}$ are WGI from cusps $m+1, \dots, n$ of $\mathcal{M}$.  
\end{theorem}

\begin{proof}
Recall from Lemma \ref{20040906} that there exists an analytic function $\Theta(s_1, \dots, s_m)$ satisfying 
\begin{equation*}
a_{01}u_1+b_{01}v_1+\cdots +a_{0h}u_{h}+b_{0h}v_{h}=\Theta(s_1, \dots, s_m)
\end{equation*}
where 
\begin{equation*}\label{19081612}
\begin{aligned}
s_i&=a_{i1}u_1+b_{i1}v_1+\cdots +a_{im}u_m+b_{im}v_m\quad (1\leq i\leq m),
\end{aligned}
\end{equation*}
and $a_{ij}=b_{ij}=0$ for $1\leq i\leq h$ and $h+1\leq j\leq m$. Also note that the Jacobian of $\log(\mathcal{X}\cap H)$ at $(0,\dots, 0)$ is of the following form 
\begin{equation}\label{19081602}
\left(\begin{array}{c|c}
\left(a_{ij} + b_{ij} \tau_{j}\right)_{1 \leq i,j \leq h} & 0 \\
\hline
* & \left(a_{ij} + b_{ij} \tau_j\right)_{h+1 \le i,j \le m}
\end{array}\right)
\end{equation}
where 
\begin{equation}\label{19082704}
A:=\begin{array}{c}
(a_{ij}+b_{ij}\tau_{j})_{1\leq i,j\leq h}
\end{array}\quad \text{and}\quad \begin{array}{c}
(a_{ij}+b_{ij}\tau_{j})_{h+1\leq i,j \leq m}
\end{array}
\end{equation}
are invertible.\footnote{If the determinant of $A$ in \eqref{19082704} is $0$, by Proposition \ref{anomalous3}, there exists an algebraic subgroup $H''$ containing $H'$ such that $\mathcal{X}\cap H''$ is an anomalous subvariety of $\mathcal{X}$. But it contradicts the assumption on $H'$. The fact that the determinant of the second matrix in \eqref{19082704} is nonzero follows from \eqref{21081315} and Proposition \ref{anomalous2}. } So if 
\begin{equation*}
\Theta(s_1, \dots, s_m):=e_1s_1+\cdots+e_ms_m+\text{higher degrees},
\end{equation*}
then 
\begin{equation*}
\begin{gathered}
\left( \begin{array}{c}
e_1   \\
\vdots   \\
e_{h}  
\end{array} \right)
=A^{-T}
\left( \begin{array}{c}
a_{01}+b_{01}\tau_1   \\
\vdots   \\
a_{0h}+b_{0h}\tau_{h} 
\end{array} \right)
\end{gathered}
\end{equation*}
and $e_{h+1}=\cdots=e_m=0$. Let $u_{h+1}=\cdots=u_m=0$. By the assumptions on $H'$, applying the same methods presented in the proof of Theorem \ref{21072807}, it is concluded that    
\begin{equation*}
v_i(u_1, \dots, u_{h}, 0, \dots, 0, u_{m+1}, \dots, u_n)\quad (1\leq i\leq h)
\end{equation*}
depends only on $u_1, \dots, u_{h}$. That is, cusps $1, \dots, h$ of $\mathcal{M}$ are WGI from cusps $m+1, \dots, n$. 
\end{proof}

\subsubsection{$l=0$}\label{20040907}
Lastly we verify $l=0$ in \eqref{20050403}. Once it is shown, the proof of Theorem \ref{19082701} will be completed by combining it with Theorems \ref{21072807} and \ref{21080805}. To prove $l=0$ in \eqref{20050403}, it is enough to demonstrate $\Theta$ in \eqref{19082903} is independent of $t_1, \dots, t_l$. Indeed, if $\Theta$ depends only on $s_1, \dots, s_m$, it means an analytic set defined by 
\begin{equation*}
\begin{gathered}
\zeta_i=a_{i1}u_1+b_{i1}v_1+\cdots +a_{im}u_m+b_{im}v_m
\end{gathered}
\end{equation*}
where $\zeta_i\in \mathbb{C} (1\leq i\leq m)$ and $\zeta_0=\Theta(\zeta_1, \dots, \zeta_m)$ is an analytic subset of $\log \mathcal{X}$ of dimension $n-m$. Said differently, a translation of an algebraic subgroup $H^{(m+1)}$ defined by 
\begin{equation*}
\begin{gathered}
M_1^{a_{i1}}L_1^{b_{i1}}\cdots M_m^{a_{im}}L_m^{b_{im}}=1,\quad (0\leq i\leq m)
\end{gathered}
\end{equation*}
contains an anomalous subvariety of $\mathcal{X}$ of dimension $n-m$. Since $H\subset H^{(m+1)}$ and each $\mathcal{X}\cap gH$ in \eqref{20050403} is a maximal anomalous subvariety of $\mathcal{X}$ of dimension $n-m-l$, $l=0$ follows.  

Now we state
\begin{proposition}\label{20040905}
$\Theta$ in \eqref{19082903} is independent of $t_1, \dots, t_l$. That is, $\Theta$ depends only on $s_1, \dots, s_m$. 
\end{proposition}

The key idea of the proof of the proposition essentially lies in the properties of $v_i$ exhibited in Theorem~\ref{potential}. According to the theorem, the degree of $u_j$ in each term of $v_i$ is \textit{even} for $i\neq j$, which implies the degree of each $u_j$ with $j>m$ in every term on the left-hand side of \eqref{19082903} must be \textit{even}. However, we show that if $\Theta$ depends on $t_i$, this is impossible. That is, the right-hand side of \eqref{19082903} necessarily contains a term in which the degree of $u_j$, for some $j>m$, is odd under the assumption. 

To make the proof simpler and easier, let us first change variables and rephrase Proposition~\ref{20040905} accordingly as follows.

By \eqref{21072901}, the Jacobian matrix of $s_1, \dots,s_m, t_1, \dots t_l$ at $(u_1, \dots, u_n)=(0, \dots, 0)$ is 
\begin{equation}\label{22031301}
\left(\begin{array}{c|c}
\bigl(a_{ij} + b_{ij} \tau_j\bigr)_{1 \leq i,j \leq m} & 0 \\
\hline
\bigl(a'_{ij} + b'_{ij} \tau_j\bigr)_{1 \leq i \leq l,\,1 \leq j \leq m} & \bigl(a'_{ij} + b'_{ij} \tau_j\bigr)_{1 \leq i \leq l,\,m+1 \leq j \leq n}
\end{array}
\right)
\end{equation}
By the assumption on $H^{(m)}$ (see the remark after Proposition \ref{anomalous3}),
\begin{equation*}\label{20042601}
\begin{gathered}
A:=\begin{array}{c}
(a_{ij}+b_{ij}\tau_j)_{1\leq i,j\leq m} 
\end{array}
\end{gathered}
\end{equation*}
is invertible, so if   
\begin{equation}\label{22031401}
\begin{gathered}
\left( \begin{array}{c}
x_1 \\
\vdots    \\
x_m    \\
\end{array} \right):=A^{-1}
\left( \begin{array}{c}
s_1  \\
\vdots    \\
s_m  \\
\end{array} \right),
\end{gathered}
\end{equation}
then each $x_i$ ($1\leq i\leq m$) is of the form
\begin{equation}\label{19090809}
\begin{gathered}
x_i=u_i+\text{higher degrees}.
\end{gathered}
\end{equation}
Adding linear combinations of $x_1, \dots, x_m$ to each $t_i$ ($1\leq i\leq l$) if necessary, we further suppose the lower left block matrix in \eqref{22031301} is zero, i.e.,   
\begin{equation}\label{22031803}
\begin{gathered}
a'_{ij}=b'_{ij}=0
\end{gathered}
\end{equation}
for $1\leq i\leq l, 1\leq j\leq m$. Since 
\begin{equation*}
\mathcal{X}\cap H\not\subset (M_k=1)
\end{equation*}
for $m+1\leq k\leq n$, equivalently, 
\begin{equation}\label{22031601}
\bold{e_k}\notin R(A')
\end{equation}
where $R(A')$ is the row vector space of $A':=\begin{array}{c}
(a'_{ij}+b'_{ij}\tau_{j})_{1\leq i\leq l,\; m+1\leq j\leq n}
\end{array}$ and $\bold{e_k}$ is a unit $1\times(n-m)$ matrix whose $k$-th entry is $1$. By \eqref{22031601}, we therefore find an invertible $(l\times l)$-matrix $L$ such that, for 
\begin{equation}\label{19081404}
\left(
\begin{array}{c}
y_1\\
\vdots\\
y_l
\end{array}
\right) 
:=L\left(
\begin{array}{c}
t_1\\
\vdots\\
t_l
\end{array}
\right), 
\end{equation}
$\left(
\begin{array}{c}
y_1\\
\vdots\\
y_l
\end{array}\right)$ is given as
\begin{equation}\label{22031801}
\left(
\begin{array}{c}
\sum_{j=m_1}^{n_1} c_{1j}u_{j}+\text{higher degrees}\\
\vdots\\
\sum_{j=m_l}^{n_l} c_{lj}u_{j}+\text{higher degrees}
\end{array}
\right)
\end{equation}
where the coefficients $c_{ij}$ satisfy the following
\begin{equation}\label{21081001}
\begin{aligned}
\bullet\;\;  &m_i<n_i\;\; \text{and} \;\;c_{im_i}, c_{in_i}\neq 0\;\;\text{for}\;\;1\leq i\leq l; \qquad \qquad \qquad \qquad \qquad \qquad \qquad\qquad \\
\bullet\;\;  &m+1\leq m_1<\cdots<m_l\;\; \text{and} \;\; n_1<\cdots<n_l\leq n. 
\end{aligned}
\end{equation}

In conclusion, by changing variables from $s_1, \dots, s_m, t_1, \dots, t_l$ to $x_1, \dots, x_m, y_1, \dots, y_l$ via \eqref{22031401}, \eqref{22031803} and \eqref{19081404}, the matrix in \eqref{22031301} is transformed into the following $(m+l)\times n$ matrix in row-echelon form
\[
\left(
\begin{array}{c|c}
I_m & 0 \\
\hline
0 &
\begin{array}{cccccc}
c_{1m_1} & \hdots & c_{1n_1} & 0 & \hdots & 0 \\
\vdots   & \ddots & \vdots   & \vdots & \ddots & \vdots \\
0 & \hdots & c_{l m_l} & \hdots & c_{l n_l} & 0
\end{array}
\end{array}
\right),
\]
which is the Jacobian of $x_1, \dots, x_m, y_1, \dots, y_l$ at $(u_1, \dots, u_n)=(0, \dots, 0)$. By abuse of notation, we rewrite \eqref{19082903} as 
\begin{equation}\label{21080801}
a_{ 01}u_1+b_{01}v_1+\cdots +a_{0m}u_m+b_{0m}v_m=\Theta(x_1, \dots, x_m, y_{1},\dots, y_{l}).
\end{equation}

Now Proposition \ref{20040905} is equivalent to 
\begin{proposition}\label{21080803}
$\Theta$ in \eqref{21080801} is independent of $y_1, \dots, y_l$. That is, $\Theta$ in \eqref{21080801} depends only on $x_1, \dots x_m$. 
\end{proposition}

Recall that the key strategy of the proof is to show that when $\Theta$ depends on $t_1, \dots, t_l$, it must, as a function of $u_1, \dots, u_n$ (via \eqref{19090809} and \eqref{22031801}), contain a term in which the exponent of some $u_j$ with $j>m$ is \textit{odd}.\footnote{This contradicts the fact that, in every term of $a_{01}u_1 + b_{01}v_1 + \cdots + a_{0m}u_m + b_{0m}v_m$, the exponent of each such $u_k$ must be even.} 

Now, once we have changed variables as above, it becomes much easier to develop the argument. Since each $y_i$ contains at least two linear terms in the variables $u_j$ (with $m+1 \le j \le n$), it is intuitively clear that the expansion of any monomial in $y_i$, viewed as a function of the $u_j$, contains a term satisfying the required condition. A rigorous proof of this fact, however, requires more laborious arguments and will be given using induction; see Claim~\ref{21080703} below.

\begin{proof}[Proof of Proposition \ref{21080803}]
Let
\begin{equation}\label{20050301}
\Theta(x_1, \dots, x_m, y_1, \dots, y_l):=\sum^{\infty}_{(\alpha_1, \dots, \alpha_m, \beta_1, \dots, \beta_l)\in\mathbb{Z}^{m+l}}c_{\alpha_1, \dots, \alpha_m, \beta_1, \dots, \beta_l}x_1^{\alpha_1}\cdots x_m^{\alpha_m}y_1^{\beta_1}\cdots y_l^{\beta_l}, 
\end{equation}
and $S$ be the set of all monomials $x_1^{\alpha_1}\cdots x_m^{\alpha_m}y_1^{\beta_1}\cdots y_l^{\beta_l}$ in \eqref{20050301} satisfying 
\begin{itemize}
\item $c_{\alpha_1, \dots, \alpha_m, \beta_1, \dots, \beta_l}\neq0$ with $(\alpha_1, \dots, \alpha_{m})\neq (0, \dots, 0)$ and $(\beta_{1}, \dots, \beta_{l})\neq (0, \dots, 0)$;
\item $\alpha_1+\cdots +\alpha_{m}$ is minimal.
\end{itemize}
Fix $(\alpha_1, \dots, \alpha_{m})$, and let $T_{(\alpha_1, \dots, \alpha_{m})}\subset S$ be the set of monomials $x_1^{\alpha_1}\cdots x_m^{\alpha_m}y_1^{\beta_1}\cdots y_l^{\beta_l}$ satisfying 
\begin{itemize}
\item every element in $T_{(\alpha_1, \dots, \alpha_{m})}$ is divisible by $x_1^{\alpha_1}\cdots x_m^{\alpha_m}$;
\item $\beta_{1}+\cdots+\beta_{l}$ is minimal.
\end{itemize} 
Define the subseries $\Theta_{T_{(\alpha_1, \cdots, \alpha_{m})}}$ of $\Theta$ by\footnote{By the definition of $T_{(\alpha_1, \cdots, \alpha_{m})}$, $\Theta_{T_{(\alpha_1, \dots, \alpha_{m})}}$ contains all the terms of $\Theta$ of the smallest degree and divisible both by $u_1^{\alpha_1} \cdots u_m^{\alpha_m}$ and some $u_j$ where $m+1\leq j \leq n$.}
\begin{equation*}
\sum_{x_1^{\alpha_1}\cdots x_m^{\alpha_m}y_1^{\beta_1}\cdots y_l^{\beta_l}\in T_{(\alpha_1, \dots, \alpha_{m})}} c_{\alpha_1, \dots, \alpha_m,  \beta_1, \dots, \beta_l}x_1^{\alpha_1}\cdots x_m^{\alpha_m} y_1^{\beta_1}\cdots y_l^{\beta_l}.
\end{equation*}
By \eqref{19090809} and \eqref{22031801}, if $\Theta_{T_{(\alpha_1, \dots, \alpha_{m})}}$ is represented as a function of $u_1, \dots, u_n$, then the following 
\begin{equation*}
\sum_{x_1^{\alpha_1}\cdots x_m^{\alpha_m}y_1^{\beta_1}\cdots y_l^{\beta_l}\in T_{(\alpha_1, \dots, \alpha_{m})}}c_{\alpha_1, \dots, \alpha_m,  \beta_1, \dots, \beta_l} u_1^{\alpha_1}\cdots u_m^{\alpha_m}\Big(\small \sum_{j=m_1}^{n_1} c_{1j}u_{j}\Big)^{\beta_1}\cdots \Big(\small \sum_{j=m_l}^{n_l} c_{1j}u_{j}\Big)^{\beta_l}
\end{equation*}
is the collection of the leading terms, which we denote by $\Theta_{T_{(\alpha_1, \dots, \alpha_{m})}}^{\mathrm{lead}}$. Since the degree of $u_j$ ($m+1\leq j\leq n$) in each term of 
\begin{equation*}
a_{01}u_1+b_{01}v_1+\cdots +a_{0m}u_m+b_{0m}v_m
\end{equation*}
is \textit{even} by Theorem \ref{potential}, the same property must be true for $\Theta$, $\Theta_{T_{(\alpha_1, \dots, \alpha_{m})}}$ and $\Theta_{T_{(\alpha_1, \dots, \alpha_{m})}}^{\mathrm{lead}}$. However, we show this is impossible in the following claim.  \\

\noindent\textbf{Convention.} For simplicity, we call a monomial of the form
\begin{equation*}
u_{1}^{\alpha_1}\cdots u_{n}^{\alpha_n}
\end{equation*}
\textit{odd} if at least one exponent $\alpha_i$ is odd. 

\begin{claim}\label{21080703}
Let 
\begin{equation*}
\begin{gathered}
z_1:=\sum_{j=m_1}^{n_1} c_{1j}u_{j}, \quad \dots, \quad z_l:=\sum_{j=m_l}^{n_l} c_{lj}u_{j}.
\end{gathered}
\end{equation*}
For $\beta\in \mathbb{N}$, any nontrivial linear combination of elements in 
\begin{equation}\label{19080105}
\{z_1^{\beta_1}\cdots z_l^{\beta_l}\;|\;\beta_1+\cdots +\beta_l=\beta\},
\end{equation}
viewed as a polynomial in $u_j$ ($m+1\leq j\leq n$), contains at least one odd monomial.  
\end{claim}
\begin{proof}[Proof of the claim]
We prove by induction on $l$. For $l=1$, clearly $z_1^{\beta}=\big(\sum_{j=m_1}^{n_1} c_{1j}u_{j}\big)^{\beta}$ contains an odd term as $z_1$ has at least two non-trivial terms (i.e. $u_{m_1}$ and $u_{n_1}$)  by the assumptions given in \eqref{21081001}. 

Suppose $l\geq 2$ and the claim is true for $1, \dots, l-1$. Let
\begin{equation}\label{21080701}
\sum_{\beta_1+\cdots +\beta_l=\beta} \tilde{c}_{\beta_1,\dots \beta_l} z_1^{\beta_1}\cdots z_l^{\beta_l}\quad (l\geq 2)
\end{equation}
be any nontrivial linear sum of elements in \eqref{19080105}. 

\begin{enumerate}
\item First, assume $\tilde{c}_{\beta,0, \dots, 0}\neq 0$ in \eqref{21080701}. If $\beta$ is odd, by \eqref{21081001}, $u_{m_1}^{\beta}$ is a non-trivial odd term appearing only in $z_1^{\beta}$, thus it appears in \eqref{21080701} as well. For $\beta$ even, we split it into the following two cases. 
\begin{enumerate}
\item If the coefficient of $z_1^{\beta-1}z_j$ in \eqref{21080701} is non-zero for some $j$ ($2\leq j\leq l$), let $k$ be the largest such $j$. Then $u_{m_1}^{\beta-1}u_{n_k}$ is odd and found only in $z_1^{\beta-1}z_k$ (again by \eqref{21081001}), hence \eqref{21080701} possesses the desired property.  
\item If the coefficient of $z_1^{\beta-1}z_j$ in \eqref{21080701} is zero for every $j$ ($2\leq j\leq l$), then an odd monomial $u_{m_1}^{r-1}u_{n_1}$ appears only in $z_1^{\beta}$, so it does in \eqref{21080701} as well. 
\end{enumerate}

\item Now suppose $\tilde{c}_{\beta,0, \dots, 0}= 0$ in \eqref{21080701}. By induction, for each $\gamma$ ($0\leq \gamma<\beta$), any linear combination of elements in $\{z_2^{\beta_2}\cdots z_{l}^{\beta_{l}}\;|\;\beta_2+\cdots +\beta_{l}=\beta-\gamma\}$, when represented as a polynomial of $u_j$ ($m+1\leq j\leq n$), has an odd term. Hence, for each $\gamma$ $(0\leq \gamma <\beta)$, a linear sum of any elements in 
\begin{equation*}
\mathcal{Z}_{\gamma}:=\{z_1^{\beta_1}\cdots z_l^{\beta_l}\;|\;\beta_1+\cdots +\beta_{l}=\beta,\quad \beta_1=\gamma\}, 
\end{equation*}
contains an odd term as well divisible by $u_{m_1}^{\gamma}$. This further implies a linear combination of elements in $\bigcup_{\gamma=0}^{\beta-1} \mathcal{Z}_{\gamma}$, when expressed as a polynomial in $u_{j}$ ($m+1\leq j\leq n$), has a non-trivial odd term. 
\end{enumerate}
This completes the proofs of Claim \ref{21080703} as well as Propositions \ref{20040905}-\ref{21080803}. 
\end{proof}
\end{proof}

\section{Main Result III}
\subsection{Verification of Theorem \ref{20042301}}\label{23080901}

In this subsection, we prove our last main result, Theorem \ref{20042301}.  For the reader’s convenience, we restate the theorem below. \\
\\
\textbf{Theorem \ref{20042301}.} \textit{Let $\mathcal{M}$ be a $2$-cusped hyperbolic $3$-manifold and $\mathcal{X}$ be its holonomy variety. If $\mathcal{X}^{oa}=\emptyset$, then either one of the following holds:
\begin{enumerate}
\item two cusps of $\mathcal{M}$ are SGI; 
\item there exists a two variable polynomial $f$ such that $\mathcal{X}$ is defined by 
\begin{equation}\label{20042304}
\begin{gathered}
f(M_1^{a}L_1^{b}M_2^{c}L_2^{d}, M_1^{md}M_2^{mb})=0, \quad f(M_1^{a}L_1^{b}M_2^{-c}L_2^{-d}, M_1^{md}M_2^{-mb})=0 
\end{gathered}
\end{equation}
for some $a,b,c,d\in \mathbb{Z}$ and $m\in \mathbb{Q}$ satisfying $mbd\neq 0$. 
\end{enumerate}}
The proof of the above theorem is based on Theorem \ref{struc} as well as various symmetric properties of $\log \mathcal{X}$ given in Theorem \ref{potential}. 

\begin{proof}
By Theorem \ref{20033003}, it is enough to show that if $\mathcal{M}$ has rationally dependent cusp shapes and the two cusps of $\mathcal{M}$ are not SGI from each other, then $\mathcal{X}$ is defined by equations given in \eqref{20042304}. 

Since $\mathcal{X}^{oa}=\emptyset$, by Theorem \ref{struc}, there exists an irreducible algebraic subgroup $H$ such that $\mathcal{X}$ is foliated by maximal anomalous subvarieties contained in   
\begin{equation*}
\bigcup_{g\in \mathcal{Z}_H}\mathcal{X}\cap gH.
\end{equation*}
Let $H$ be defined by 
\begin{equation}\label{20031601}
\begin{gathered}
M_1^{a_1}L_1^{b_1}M_2^{c_1}L_2^{d_1}=M_1^{a_2}L_1^{b_2}M_2^{c_2}L_2^{d_2}=1.
\end{gathered}
\end{equation}
By changing basis if necessary, we assume \eqref{20031601} is of the following form
\begin{equation*}
\begin{gathered}
M_1^{a_1}L_1^{b_1}M_2^{c_1}L_2^{d_1}=M_1^{a_2}M_2^{c_2}=1.
\end{gathered}
\end{equation*}
Then $\mathcal{X}\cap H$ is locally biholomorphic to 
\begin{equation*}
\begin{gathered}
a_1u_1+b_1v_1+c_1u_2+d_1v_2=a_2u_1+c_2u_2=0.
\end{gathered}
\end{equation*}
 If $\mathcal{X}\cap gH$ is an anomalous subvariety of $\mathcal{X}$ for infinitely many $g\in \mathcal{Z}_H$, equivalently, the following equations
\begin{equation*}
a_1u_1+b_1v_1+c_1u_2+d_1v_2=\xi_1, \quad a_2u_1+c_2u_2=\xi_2
\end{equation*}
define a $1$-dimensional complex manifold for infinitely many $\xi_1, \xi_2\in\mathbb{C}$. 
Thus there exists a holomorphic function $h$ such that  
\begin{equation}\label{19090101}
a_1u_1+b_1v_1+c_1u_2+d_1v_2=h(a_2u_1+c_2u_2).
\end{equation}
If $b_1=0$ (resp. $d_1=0$), then one can easily check that $d_1=0$ (resp. $b_1=0$) and 
\begin{equation*}
a_1u_1+c_1u_2=m(a_2u_1+c_2u_2)
\end{equation*}
for some $m\in \mathbb{Q}\backslash\{0\}$. But this contradicts the fact that $H$ is an algebraic subgroup of dimension $2$. Without loss of generality, we assume $b_1, d_1\neq 0$, and claim
\begin{claim}
\begin{equation*}
(c_2, a_2)=m(b_1, d_1)
\end{equation*}
for some $m\in\mathbb{Q}\backslash \{0\}$. 
\end{claim} 
\begin{proof}[Proof of the claim]
Let $h(t)$ in \eqref{19090101} be defined by $\sum^{\infty}_{i=1}e_it^{2i-1}$ and so 
\begin{equation}\label{19072805}
\begin{aligned}
a_1u_1+b_1v_1+c_1u_2+d_1v_2&=\sum^{\infty}_{i=1}e_i(a_2u_1+c_2u_2)^{2i-1}\\
&=\sum^{\infty}_{i=1} e_i\Bigg(\sum^{2i-1}_{j=0}\binom{2i-1}{j}a_2^{2i-1-j}c_2^ju_1^{2i-1-j}u_2^{j}\Bigg).
\end{aligned}
\end{equation}
By Theorem \ref{potential}, since the degree of $u_i$ (resp. $u_{i+1}$) in every term of $v_i$ is odd (resp. even), we split \eqref{19072805} as follows:
\begin{equation}\label{21072401}
\begin{gathered}
a_1u_1+b_1v_1=\sum^{\infty}_{i=1}e_i\Bigg(\sum^{2i-2}_{j=0, even}\binom{2i-1}{j}a_2^{2i-1-j}c_2^ju_1^{2i-1-j}u_2^{j}\Bigg),\\
c_1u_2+d_1v_2=\sum^{\infty}_{i=1}e_i\Bigg(\sum^{2i-2}_{j=0, even}\binom{2i-1}{j}
c_2^{2i-1-j}a_2^ju_2^{2i-1-j}u_1^{j}\Bigg).
\end{gathered}
\end{equation}
As $\dfrac{1}{2}\dfrac{\partial\Phi}{\partial u_1}=v_1$ and $\dfrac{1}{2}\dfrac{\partial\Phi}{\partial u_2}=v_2$ by Theorem \ref{potential}, we get 
\begin{equation}\label{20031602}
\frac{1}{b_1}\frac{1}{2i-j}e_i\binom{2i-1}{j}a_2^{2i-1-j}c_2^ju_1^{2i-j}u_2^{j}=\frac{1}{d_1}\frac{1}{2i-k}e_i\binom{2i-1}{k}c_2^{2i-1-k}a_2^{k}u_2^{2i-k}u_1^{k}
\end{equation}
for all $i, j, k$ such that $j+k=2i\;(\geq 4)$ from \eqref{21072401}. Now \eqref{20031602} implies
\begin{equation*}
\begin{aligned}
\frac{1}{b_1}\frac{1}{2i-j}\binom{2i-1}{j}a_2^{2i-1-j}c_2^j&=\frac{1}{d_1}\frac{1}{2i-k}\binom{2i-1}{k}c_2^{2i-1-k}a_2^{k}\\
\Longrightarrow
\frac{1}{k}\binom{2i-1}{j}d_1c_2&=\frac{1}{j}\binom{2i-1}{k}b_1a_2\\
\Longrightarrow 
\frac{1}{k}\frac{1}{j!(k-1)!}d_1c_2&=\frac{1}{j}\frac{1}{k!(j-1)!}
b_1a_2\quad (\text{by }j+k=2i)\\
\Longrightarrow
d_1c_2&=b_1a_2.
\end{aligned}
\end{equation*}
This completes the proof of the claim. 
\end{proof}
By Theorem \ref{potential}, the degree of $u_1$ (resp. $u_2$) in each term of $v_2$ is even (resp. odd) and so \eqref{19072805} implies 
\begin{equation}\label{20031605}
a_1u_1+b_1v_1-c_1u_2-d_1v_2=\sum^{\infty}_{i=1}e_i(a_2u_1-c_2u_2)^{2i-1}=h(a_2u_1-c_2u_2).
\end{equation}
Let $\mathcal{C}:=\mathcal{X}\cap (M_2=L_2=1)$ and $\mathcal{C}'$ be the image of $\mathcal{C}$ under the following transformation:
\begin{equation*}
M'_1:=M_1^{a_1}L_1^{b_1}, \quad L'_1:=M_1^{a_2}.
\end{equation*} 
By projecting onto the first two coordinates if necessary, we consider $\mathcal{C}'$ as an algebraic curve in $\mathbb{G}^2(:=(M'_1, L'_1))$. Let $f(M'_1, L'_1)=0$ be the defining equation of $\mathcal{C}'$, which is, near $(1,1)$, locally biholomorphic to $v'_1=h(u'_1)$ where $u'_1:=\log M'_1, v'_1:=\log L'_1$. Then \eqref{19090101} (resp. \eqref{20031605}) is equivalent to 
\begin{equation}\label{20042603}
f(M_1^{a_1}L_1^{b_1}M_2^{c_1}L_2^{d_1}, M_1^{a_2}M_2^{c_2})=0 \quad (\text{resp. }f(M_1^{a_1}L_1^{b_1}M_2^{-c_1}L_2^{-d_1}, M_1^{a_2}M_2^{-c_2})=0).
\end{equation} 
Since $v_1$ and $v_2$ are determined by \eqref{19090101} and \eqref{20031605}, $\mathcal{X}$ is defined by the two equations in \eqref{20042603}. 
\end{proof}

\subsection{Application of Theorem \ref{20042301}}\label{exam}

As a corollary of the above theorem, we have the following, which provides a useful criterion to check whether $\mathcal{X}^{oa}=\emptyset$ or not for the holonomy variety $\mathcal{X}$ of a $2$-cusped hyperbolic $3$-manifold:
\begin{corollary}\label{23081005}
Let $\mathcal{M}$ be a $2$-cusped hyperbolic $3$-manifold whose cusps are not SGI and $\mathcal{X}$ be its holonomy variety. For the Neumann-Zagier potential function $\mathcal{M}$ given as 
\begin{equation}\label{23081001}
\Phi(u_1, u_2):=\sum_{i,j:even} c_{i,j}u_1^iu_2^j, 
\end{equation}
if $36c_{4,0}c_{0,4}\neq c_{2,2}^2$, then $\mathcal{X}^{oa}\neq \emptyset$; that is, $\mathcal{X}$ contains only finitely many anomalous subvarieties. 
\end{corollary}
\begin{proof}
On the contrary, we assume $\mathcal{X}^{oa}=\emptyset$. Since two cusps of $\mathcal{M}$ are not SGI from each other, we fall into the second case of Theorem \ref{20042301}. That is, if $\mathcal{X}$ is defined by equations as given in \eqref{20042304}. Equivalently, by the proof of Theorem \ref{20042301}, this implies 
\begin{equation}\label{23081003}
a_1u_1+b_1v_1+c_1u_2+d_1v_2=\sum^{\infty}_{i=1}e_i(mb_1u_1+md_1u_2)^{2i-1}, 
\end{equation} 
for some $e_i\in \mathbb{C}$. Since
\begin{equation*}
\begin{gathered}
v_1=\frac{1}{2}\frac{\partial \Phi}{\partial u_1}=\frac{i}{2}\sum_{i,j\;:\:even} c_{i,j}u_1^{i-1}u_2^j, \quad 
v_2=\frac{1}{2}\frac{\partial \Phi}{\partial u_2}=\frac{j}{2}\sum_{i,j\;:\:even} c_{i,j}u_1^{i}u_2^{j-1}
\end{gathered}
\end{equation*}
by \eqref{23081001}, combining with \eqref{23081003}, it follows that
\begin{equation*}
\begin{aligned}
b_1(2c_{4,0}u_1^{3}+c_{2,2}u_1u_2^2)+d_1(2c_{0,4}u_2^{3}+c_{2,2}u_1u_2^2)&=e_2(mb_1u_1+md_1u_2)^3\\
&=e_2m^3(b_1^3u_1^3+3b_1^2d_1u_1^2u_1+3b_1d_1^2u_1u_2^2+d_1^3u_2^3)\\
\end{aligned}
\end{equation*}
implying $c_{4,0}=\frac{e_2m^3b_1^2}{2}$, $c_{0,4}=\frac{e_2m^3d_1^2}{2}$ and $c_{2,2}=3e_2m^3b_1d_1$. In conclusion, $36c_{4,0}c_{0,4}= c_{2,2}^2$ follows. 
\end{proof}

The above corollary is very practical and applies broadly to any $2$-cusped hyperbolic $3$-manifold whose two cusps are not SGI to each other. In fact, we have applied the criterion in the corollary to more than a hundred manifolds in the SnapPy census and verified that they all satisfy the condition.

If there exists a $2$-cusped manifold whose Neumann–Zagier potential function has vanishing terms of homogeneous degree $4$, then one can analyze the coefficients of the next nonvanishing terms and derive an analogous criterion, applying it accordingly. (However, among the two-cusped hyperbolic $3$-manifolds in the SnapPy census \cite{CDGW}, if the cusps are not SGI to each other, then most have nonvanishing terms of homogeneous degree $4$ in their potential functions. The only exception we have found so far is the manifold $v2788$. See \cite{jeon6} and \cite{JO1} for more details.)

In what follows, we present a well-known concrete example, for which not only the above corollary applies, but the core techniques of the paper yield a complete analysis of the anomalous subvarieties of its holonomy variety. In \cite{AD}, J. Aaber and N. Dunfield studied the complement of the $(-2,3,8)$-pretzel link $\mathcal{W}$, a sibling of the Whitehead link complement, and computed its Neumann–Zagier potential function as follows:
\begin{equation}\label{23081101}
\Phi(u_1, u_2) = \frac{\sqrt{-1}}{2} u_1^2+\frac{\sqrt{-1}}{2}u_2^2+\frac{-3+\sqrt{-1}}{192}u_1^4-\frac{1+\sqrt{-1}}{32}u_1^2u_2^2+\frac{-3+\sqrt{-1}}{192}u_2^4+\cdots.
\end{equation}

First, as $36\big(\frac{-3+\sqrt{-1}}{192}\big)^2\neq \big(\frac{1+\sqrt{-1}}{32}\big)^2$, the holonomy variety $\mathcal{X}$ of $\mathcal{W}$ satisfies $\mathcal{X}^{oa}\neq \emptyset$ by Corollary \ref{23081005}. That is, $\mathcal{X}$ has only finitely many anomalous subvarieties. 

It is known that the group of self-isometries of $\mathcal{W}$ contains at least four non-trivial elements \cite{AD}: 
$$\{\mathfrak{m_1}\rightarrow \pm \mathfrak{m_2},  \mathfrak{l_1}\rightarrow \pm \mathfrak{l_2}\}\quad \text{and}\quad \{\mathfrak{m_1}\rightarrow \pm \mathfrak{l_2},  \mathfrak{l_1}\rightarrow \mp \mathfrak{m_2}\}.$$ Consequently, following the same analysis given in Section \ref{23080305}, one can check that $\mathcal{X}$ also has at least four anomalous subvarieties contained either in 
$$M_1=M_2^{\pm 1}, L_1=L_2^{\pm 1}\quad \text{or}\quad  M_1=L_2^{\pm 1},L_1=M_2^{\mp 1}.$$ In the proposition below, using \eqref{23081101}, we describe all potential anomalous subvarieties of $\mathcal{W}$ including the aforementioned ones.  

\begin{proposition}\label{23081220}
Let $\mathcal{W}$ and $\mathcal{X}$ be the same as above. Let $H$ be an algebraic subgroup such that $\mathcal{X}\cap H$ is a $1$-dimensional anomalous subvariety of $\mathcal{X}$. Then $H$ is either one of the following eight algebraic subgroups:
\begin{equation}\label{23081210}
\begin{gathered}
(M_i=L_i=1,\;i=1,2), \quad (M_1=M_2^{\pm 1}, L_1=L_2^{\pm 1}),\quad(M_1=L_2^{\pm 1}, L_1=M_1^{\mp 1})\\
(M_1=M_2^{\pm 2}L_2^{\pm 1}, L_1=M_2^{\mp 1}L_2^{\pm 2})\quad \text{and}\quad
(M_1^{5}=M_2^{\pm 2}L_2^{\pm 1}, L_1^5=M_2^{\mp 1}L_2^{\pm 2}).
\end{gathered}
\end{equation}
\end{proposition}
\begin{proof}
Suppose $H$ is defined by 
\begin{equation*}
M_1^{a_j}L_1^{b_j}M_2^{c_j}L_2^{d_j}=1, \quad (j=1,2).
\end{equation*}
We further assume that $H$ is defined by neither $M_1=L_1=1$ nor $M_2=L_2=1$. This implies that both 
$$\det \left(\begin{array}{cc}
a_1 & b_1\\
a_2 & b_2
\end{array}\right)\quad \text{and}\quad \det \left(\begin{array}{cc}
c_1 & d_1\\
c_2 & d_2
\end{array}\right)$$ 
are invertible. To simplify the proof, suppose 
$$\left(\begin{array}{cc}
a_1 & b_1\\
a_2 & b_2
\end{array}\right)=-I\quad\text{and}\quad \left(\begin{array}{cc}
c_1 & d_1\\
c_2 & d_2
\end{array}\right)\in\text{GL}_2(\mathbb{Q}).$$ 
Since $\mathcal{X}\cap H$ is an anomalous subvariety of $\mathcal{X}$, equivalently, it means 
\begin{equation}\label{23081102}
u_1=c_1u_2+d_1v_2, \quad v_1=c_2u_2+d_2v_2
\end{equation}
is an analytic manifold of dimension $1$. As 
\begin{equation*}
v_1=\sqrt{-1}u_1+\frac{-3+\sqrt{-1}}{48}u_1^3-\frac{1+\sqrt{-1}}{16}u_1u_2^2+\cdots
\end{equation*}
by \eqref{23081101}, combining with \eqref{23081102}, the following identity is induced:
\begin{equation}\label{23081103}
\begin{gathered}
\sqrt{-1}(c_1u_2+d_1v_2)+\frac{-3+\sqrt{-1}}{48}(c_1u_2+d_1v_2)^3-\frac{1+\sqrt{-1}}{16}(c_1u_2+d_1v_2)u_2^2+\cdots\\
=c_2u_2+d_2v_2.
\end{gathered}
\end{equation}
Using the first equation in \eqref{23081102} along with  
\begin{equation*}
\begin{gathered}
v_2=\sqrt{-1}u_2+\frac{-3+\sqrt{-1}}{48}u_2^3-\frac{1+\sqrt{-1}}{16}u_2u_1^2+\cdots
\end{gathered} 
\end{equation*}
(again derived from \eqref{23081101}), if $u_1$ is represented as a function of $u_2$, it is of the following form:
\begin{equation}\label{23081201}
u_1=(c_1+d_1\sqrt{-1})u_2+\text{higher terms}.
\end{equation}
Now plugging \eqref{23081201} into \eqref{23081103}, we consider \eqref{23081103} as the identity between two functions of $u_2$ and examine the coefficients of each term in it. 

First, comparing the coefficients of $u_2$ in \eqref{23081103}, it follows that 
\begin{equation*}
\sqrt{-1}(c_1+d_1\sqrt{-1})=c_2+d_2\sqrt{-1}\Longrightarrow c_1=d_2, -d_1=c_2. 
\end{equation*}
Computing the coefficients of $u_2^3$ in \eqref{23081103}, we obtain
\begin{equation*}
\begin{gathered}
d_1\sqrt{-1}\Big(\frac{-3+\sqrt{-1}}{48}-\frac{1+\sqrt{-1}}{16}z^2\Big)+\frac{-3+\sqrt{-1}}{48}z^3-\frac{1+\sqrt{-1}}{16}z\\
=c_1\Big(\frac{-3+\sqrt{-1}}{48}-\frac{1+\sqrt{-1}}{16}z^2\Big)
\end{gathered}
\end{equation*}
where $z=c_1+d_1\sqrt{-1}$, which is simplified to
\begin{equation}\label{26042201}
\begin{aligned}
\frac{-3+\sqrt{-1}}{48}(z^3-\overline{z})=\frac{1+\sqrt{-1}}{16}(z-\overline{z}z^2).
\end{aligned}
\end{equation}
Expanding \eqref{26042201}, we get 
\begin{equation*}
\begin{gathered}
(-1+2\sqrt{-1})\big((c_1^3-3c_1d_1^2-c_1)+(3c_1^2d_1-d_1^3+d_1)\sqrt{-1}\big)\\
=3(-c_1^3-c_1d_1^2+c_1)+3(-c_1^2d_1-d_1^3+d_1)\sqrt{-1}.
\end{gathered}
\end{equation*}
Since $c_1, d_1\in \mathbb{Q}$, it implies
\begin{equation}\label{26042101}
\begin{aligned}
-(c_1^3-3c_1d_1^2-c_1)-2(3c_1^2d_1-d_1^3+d_1)&=3(-c_1^3-c_1d_1^2+c_1)\\
\Longrightarrow 
-c_1^3-3c_1d_1^2+c_1+3c_1^2d_1-d_1^3+d_1&=0
\end{aligned}
\end{equation}
and 
\begin{equation}\label{26042102}
\begin{aligned}
2(c_1^3-3c_1d_1^2-c_1)-(3c_1^2d_1-d_1^3+d_1)&=3(-c_1^2d_1-d_1^3+d_1)\\
\Longrightarrow 
c_1^3-3c_1d_1^2-c_1+2d_1^3-2d_1&=0.
\end{aligned}
\end{equation}
By adding and subtracting \eqref{26042101} and \eqref{26042102}, it follows that 
\begin{equation}\label{26042105}
-6c_1d_1^2+3c_1^2d_1+d_1^3-d_1=-2c_1^3+2c_1+3c_1^2d_1-3d_1^3+3d_1=0. 
\end{equation} 
\begin{enumerate}
\item If $d_1=0$, then $c_1$ is either $0$ or $\pm 1$ by \eqref{26042105}. Thus \eqref{23081102} is equivalent to either $u_1=v_1=0$ or $u_1=\pm u_2, v_1=\pm v_2$ respectively, and $H$ is given by either $M_1=L_1=1$ or $M_1=M_2^{\pm1}, L_1=L_2^{\pm1}$ accordingly. 

\item Suppose $d_1\neq 0$. Then 
\begin{equation}\label{26042107}
d_1^2-1=6c_1d_1-3c_1^2
\end{equation}
by the first equation in \eqref{26042105}. Substituting it into the second equation in \eqref{26042105}, it follows that 
\begin{equation*}
\begin{aligned}
-2c_1^3+2c_1+3c_1^2d_1-3d_1(6c_1d_1-3c_1^2)=c_1(-c_1^2+1+6c_1d_1-9d_1^2)=0. 
\end{aligned}
\end{equation*}
\begin{enumerate}
\item If $c_1=0$, then $d_1=\pm 1$ from \eqref{26042105}. Thus \eqref{23081102} is equivalent to $u_1=\pm v_2, v_1=\mp u_2$, which means $H$ is defined by $M_1=L_2^{\pm 1}, L_1=M_2^{\mp 1}$. 

\item If $c_1\neq 0$, then $-c_1^2+1+6c_1d_1-9d_1^2=0$. Combining this with \eqref{26042107}, we obtain $c_1=\pm 2d_1$. So $(c_1, d_1)$ is either $(\pm 2, \pm1)$ or $(\pm \frac{2}{5}, \pm \frac{1}{5})$, and \eqref{23081102} is either
$$u_1=\pm 2u_2\pm 1v_2, \quad v_1=\mp u_2\pm 2v_2$$
or 
$$u_1=\pm\frac{2}{5}u_2\pm\frac{1}{5}v_2, \quad v_1=\mp\frac{1}{5}u_2\pm\frac{2}{5}v_2$$
respectively. Consequently, $H$ is one of the last two cases in \eqref{23081210}. 
\end{enumerate}
\end{enumerate}
\end{proof}

As explained above and in Section \ref{23080305}, the algebraic subgroups in the first line of \eqref{23081210} clearly generate anomalous subvarieties of $\mathcal{X}$; indeed, these are the only ones that exhibit this property. For the algebraic subgroups in the second line, using the defining equations of $\mathcal{X}$ given in (5.10) of \cite{AD}, we have verified that they do not produce anomalous subvarieties of $\mathcal{X}$.\footnote{For instance, if the defining equations of $\mathcal{X}$ are
$$
F(M_1, M_2, L_1)=F(M_2, M_1, L_2)=0 
$$
where $F$ is the polynomial given in (5.10) of \cite{AD}, one can show that 
$$F(M_2^2L_2, M_2, M_2^{-1}L_2^2)=0\quad \text{and}\quad  F(M_2, M_2^2L_2, L_2)=0$$ 
do not share any common components, implying 
$$\mathcal{X}\cap (M_1=M_2^2L_2,\;L_1=M_2^{-1}L_2^2) $$
is not a $1$-dimensional anomalous subvariety of $\mathcal{X}$.}

As noted earlier, the Whitehead link complement also possesses this property, as does any 2-cusped manifold admitting a self-isometry that swaps the two cusps. In fact, many 2-cusped \textit{arithmetic} 3-manifolds also appear to possess this property. We will explore these examples, along with others, in more detail in \cite{JO1}.

Beyond the $2$-cusped case, the Borromean rings complement and the complements of chain links are classical examples that admit nontrivial symmetries between their components \cite{NR2, rat, thu}. Consequently, their holonomy varieties contain anomalous subvarieties of a similar type, arising from these symmetries, as described in Proposition~\ref{23081220}.

\vspace{5 mm}
Department of Mathematics, POSTECH\\
77 Cheong-Am Ro, Pohang, South Korea\\
\\
\emph{Email Address}: bogwang.jeon@postech.ac.kr
\end{document}